\documentclass[a4paper,10pt,reqno]{amsart}

\newtheorem{thm}{Theorem}[section]
\newtheorem{lem}[thm]{Lemma}
\newtheorem{cor}[thm]{Corollary}
\newtheorem{prop}[thm]{Proposition}
\theoremstyle{definition}
\newtheorem{defn}[thm]{Definition}
\newtheorem{fact}[thm]{Fact}
\newtheorem{rem}[thm]{Remark}

\newcommand{\C}{{\mathcal C}}

\begin{document}

\title[Online containers for hypergraphs, with applications]{Online
  containers for hypergraphs, with applications to linear equations}

\date{26th June 2015, 11th March 2016}

\author{David Saxton}
\author{Andrew Thomason}

\address{Department of Pure Mathematics and Mathematical Statistics\\
Centre for Mathematical Sciences, Wilberforce Road, Cambridge CB3 0WB, UK}

\email{dwsaxton@gmail.com}
\email{a.g.thomason@dpmms.cam.ac.uk}

\begin{abstract}
  A set of containers for a hypergraph $G$ is a collection $\mathcal{C}$ of
  vertex subsets, such that for every independent (or, indeed, merely
  sparse) set $I$ of $G$ there is some $C\in\mathcal{C}$ with $I\subset C$, no
  member of $\mathcal{C}$ is large, and the collection $\C$ is relatively
  small. Containers with useful properties have been exhibited by Balogh,
  Morris and Samotij~\cite{BMS} and by the authors~\cite{ST1,ST2,ST3}, along
  with several applications.

  Our purpose here is to give a simpler algorithm than the one used
  in~\cite{ST2}, which nevertheless yields containers with all the
  properties needed for the main container theorem of~\cite{ST2} and its
  consequences. Moreover this algorithm produces containers having the
  so-called online property, allowing the colouring results of~\cite{ST2}
  to be extended to all, not just simple, hypergraphs.
%% added weaker
  Most of the proof of the container theorem remains the same if this new
  algorithm is used, and we do not repeat all the details here, but
  describe only the changes that need to be made. However, for illustrative
  purposes, we do include a complete proof of a slightly weaker but simpler
  version of the theorem, which for many (perhaps most) applications is
  plenty.

  We also present applications to the number of solution-free sets of
  linear equations, including the number of Sidon sets, that were
  announced in~\cite{ST2}.
\end{abstract}

\maketitle

\section{Introduction}

Let $G$ be an $r$-uniform hypergraph with vertex set $V(G)$ and edge set
$E(G)$. Very often we shall assume that $V(G)$ is the set
$[n]=\{1,\ldots,n\}$. A subset $I\subset V(G)$ is independent if it
contains no edge. A set of containers for $G$ is a collection $\mathcal{C}$
of subsets of $V(G)$, such that, for every independent set~$I$, there is a
container $C\in\mathcal{C}$ with $I\subset C$. To be useful, each container
should be not much bigger than an independent set can be, but the number of
containers should be much smaller than the number of independent sets. A
collection $\mathcal{C}$ can sometimes serve as a substitute for the
collection of independent sets in simple expectation arguments, the small
size of $|\mathcal{C}|$ rendering the argument effective where an
expectation instead over all independent sets would yield nothing
worthwhile.

Saphozhenko~\cite{Sap1,Sap2} seems to have been the first to explicitly
consider containers for ordinary graphs, in his studies of the number of
independent sets in regular graphs, and he coined the phrase ``container
method'' (see~\cite{Sap5}). The usefulness of containers for a particular
3-uniform hypergraph was highlighted by Green~\cite{G0} in his solution to
the Cameron-Erd\H{o}s problem: see~\S\ref{subsec:lineq} for more on
this. Containers for simple regular hypergraphs were introduced
in~\cite{ST1} (extended to non-regular in~\cite{ST3}). More recently,
however, the containers constructed by Balogh, Morris and Samotij
in~\cite{BMS} (inspired originally by the graph methods of Kleitman and
Winston~\cite{KW}) and by the authors in~\cite{ST2} have been especially
effective in addressing certain questions of extremal combinatorics, due to
the small size (essentially optimal) of the collection~$\mathcal{C}$. In
consequence, the method has been adopted more widely, as in for example
\cite{BMT,CM,KOTZ,NS}.

The purpose of the present paper is to describe an algorithm for container
construction that is similar to, but different from, the algorithm
in~\cite{ST2}. The algorithm here is more straightforward (it passes
through the vertex set only once instead of multiple times), and so it is
more transparent and comprehensible. It has the further advantage that it
has the so-called online property; the algorithm in~\cite{ST2} had this
property only when applied to simple hypergraphs. The online property is
needed for applications where the number of vertices $|C|$ of a
container~$C$ is important (many applications care instead about the number
of edges inside~$C$). This is described a little more
in~\S\ref{subsec:containers}, and an application to list colourings is
given in~\S\ref{subsec:listcol}.

The paper~\cite{ST2} has a fair amount of discussion of the container
method and of the motivation behind that algorithm. Essentially all of what
is written there is relevant here too, so rather than reproduce it we refer
the reader to~\cite{ST2} for fuller information. The statement of the main
container results here, Theorem~\ref{thm:cover} and its corollaries, stated
in~\S\ref{subsec:containers}, are the same as in~\cite{ST2} (apart from the
replacement of tuples $(T_{r-1},\ldots,T_0)$ by single sets~$T$ --- see
Remark~\ref{rem:enlarge}). Indeed, the proofs from~\cite{ST2} carry over
word for word, once some straightforward degree calculations have been
carried out verifying that the present algorithm performs as least as well
as the old one. For this reason we do not give the full proofs, but
restrict ourselves just to carrying out these degree calculations and
explaining why the rest of the proof is identical apart from purely
cosmetic differences. This approach means the present paper is not
completely self-contained, but it avoids excessive duplication.

Having said that, much of the complexity of the proofs in~\cite{ST2} arose
from an effort to establish a good bound on the measure $\mu(C)$ of each
container (stated in Theorem~\ref{thm:cover}~(d)). For most applications, it
is enough that $\mu(C)$ is bounded merely by some constant less than
one. It is much easier to prove such a weaker result, and so, for
illustrative purposes, we include a full proof of a slightly weaker version
of the main container theorem; this is Theorem~\ref{thm:coverweak}.

During its operation, the algorithm used here monitors certain quantities
(subset degrees), and takes certain actions when these quantities reach a
certain threshold. These thresholds are specified by {\em threshold
  functions} $\theta_s$, discussed in~\S\ref{subsec:actual}. The threshold
functions are evaluated before the algorithm starts (this is a crucial
difference between the present algorithm and that in~\cite{ST2}). The
values of these functions determine how small the containers are and, more
importantly, what conditions on the hypergraph are needed in order to build
the containers. To obtain Theorem~\ref{thm:cover} we need to choose
$\theta_s$ quite carefully, so that the action of this algorithm emulates
the one in~\cite{ST2}. But, if a weaker result is acceptable, then such
delicacy is not necessary, and a more straightforward choice is enough to
give the (still very effective) Theorem~\ref{thm:coverweak}. More
explanation can be found in~\S\ref{subsec:actual} and~\S\ref{sec:analysis}.

At the same time as giving the new algorithm, we take the opportunity to
describe some applications concerning solution-free sets for linear
equations: see~\S\ref{subsec:lineq}. These were announced in~\cite{ST2} but
no proofs were given.

\subsection{Containers}\label{subsec:containers}

The main container theorem is Theorem~\ref{thm:cover} below. We define here
the terms needed, but for more discussion of them we refer the reader
to~\cite{ST2}. 

The fundamental point of the construction is that there is a function
$C:\mathcal{P}[n]\to\mathcal{P}[n]$, defined by means of an algorithm,
which, given any small set $T$ as input, produces some larger set
$C(T)$. The algorithm also ensures that, for every independent set~$I$,
there is some small $T\subset I$ with $I\subset C(T)$. This gives a
collection of containers $\mathcal{C}=\{C(T): T\mbox{ is small}\}$, where
$\mathcal{C}$ is relatively small because there are few small sets~$T$.

Of course, the terms all need to be quantified. We measure set sizes using
{\em degree measure}, the measure of $S\subset[n]$ being
$\mu(S)=(1/nd)\sum_{u\in U}d(u)$. Here $d(u)$ is the degree of the
vertex~$u$ in $G$ and $d$ is the average degree. When the graph is regular
then degree measure agrees with the uniform measure $|S|/n$, but for
non-regular graphs degree measure is more useful.

The reason for using degree measure is that, in a general $r$-uniform
hypergraph of average degree~$d$, the size of an independent set can be
arbitrarily close to~1 in uniform measure, and so uniform measure will
give no useful bound on container sizes. However (as is easily shown) the
degree measure of an independent set is at most $1-1/r$, and so there is
hope of bounding the degree measure of the containers away from~1. This
is the import of Theorem~\ref{thm:cover}~(d). The algorithm is thus
designed with degree measure in mind. Consequently the sets $T$ it produces
have $\mu(T)$ small (Theorem~\ref{thm:cover}~(b)). But for $|\mathcal{C}|$
to be relatively small we need to ensure $|T|$ is small for these
sets~$T$. The algorithm achieves this by the simple expedient of not
placing vertices of very small degree into~$T$: thus $\mu(T)$ small implies
$|T|$ is small (Theorem~\ref{thm:cover}~(c)).

As usual we define $[n]^{(s)}=\{\sigma\subset [n]:|\sigma|=s\}$ and
$[n]^{(\le s)}=\bigcup_{t\le s}[n]^{(t)}$. The degree $d(\sigma)$ of a
subset $\sigma$, where $|\sigma|\le r$, is the number of edges of $G$ that
contain~$\sigma$. We make frequent use of the definition
$d^{(j)}(\sigma)=\max\{d(\sigma'): \sigma\subset\sigma'\in[n]^{(j)}\}$,
though usually we write $d^{(j)}(v)$ instead of $d^{(j)}(\{v\})$.

A parameter $\tau$ appears in all the theorems and in the
algorithm. Roughly speaking, $\tau$ is the measure of the sets $T$, and the
smaller $\tau$ is, the smaller is $\mathcal{C}$. In
Theorem~\ref{thm:cover}, how small $\tau$ can be is determined by the {\em
  co-degree function} $\delta(G,\tau)$, defined by
$$
\delta(G,\tau)\,=\,2^{\binom{r}{2}-1}\sum_{j=2}^r\,
2^{-\binom{j-1}{2}}\delta_j
\quad{\rm where}\quad
\delta_j\,\tau^{j-1}nd\,
=\,\sum_v \,d^{(j)}(v)\,,
\quad 2\le j\le r\,.
$$
This function is identical to the one in~\cite{ST2}. Note that
$\delta(G,\tau)$ is decreasing in~$\tau$, and hence the condition
$\delta(G,\tau)\le\zeta$, which appears in Theorem~\ref{thm:cover}, is
really a lower bound on~$\tau$. The parameter $\zeta$ can be chosen to
suit, a typical value being $\zeta=1/12r!$. The function $\delta(G,\tau)$
depends on the quantities $d^{(j)}(v)$; the larger the subset degrees are
in $G$ then the larger $d^{(j)}(v)$ is likely to be, and hence $\tau$ must
be larger to achieve the bound $\delta(G,\tau)\le\zeta$.
The relationship between the subset degrees and $\tau$ is thus implicit in
the function $\delta(G,\tau)$.

We now state the main theorem.

\begin{thm}\label{thm:cover}
  Let $G$ be an $r$-graph with vertex
  set~$[n]$. Let
  $\tau,\zeta>0$ satisfy $\delta(G,\tau)\le\zeta$. Then there
  is a function $C:\mathcal{P}[n]\to\mathcal{P}[n]$, such that, for every
  independent set $I\subset[n]$ there exists $T\subset I$ with
 \begin{itemize}
 \item[(a)] $I \subset C(T)$,
 \item[(b)] $\mu(T) \le 2r\tau/\zeta$, 
 \item[(c)] $|T|\le 2r\tau n/\zeta^2$,  and
 \item[(d)] $\mu(C(T)) \le 1 - 1/r! + 4\zeta+2r\tau/\zeta$.
 \end{itemize}
 Moreover
 $C(T)\cap[w]=C(T\cap[w])\cap[w]$ for all $T\in\mathcal{P}[n]$ and  $w\in[n]$.

 Indeed, the above holds for all sets $I\subset[n]$
 such that either $G[I]$ is $\lfloor\tau^{r-1} \zeta e(G)/n\rfloor$-degenerate
 or $e(G[I])\le 2r \tau^re(G)/ \zeta$.
\end{thm}
As stated previously, this theorem is nearly the same as
\cite[Theorem~3.4]{ST2}. The numerical expressions that appear are exactly
the same. The two differences are that a tuple $(T_{r-1},\ldots,T_0)$ has
been replaced by a single set~$T$ (which, though not a strengthening, does
make the theorem easier on the eye; see Remark~\ref{rem:enlarge}), and that
the online property now holds for all $r$-graphs rather than just for simple
ones.

The online property is the statement that $C(T)\cap[w]=C(T\cap[w])\cap[w]$
for all $w\in[n]$. One way to think of it is like this. Suppose the labelling of
the vertices was hidden initially, but was then revealed one vertex at a
time. If the algorithm which has to build a container $C(T)$ from $T$ has
the online property, then it will decide which members of $[w]$ will lie in
$C(T)$ just from knowing the set $T\cap[w]$. To this extent the algorithm
is behaving like an online algorithm, though it needs to be remembered that
the whole graph is known before the vertex labelling is revealed.

The online property is needed when the number of vertices $|C|$ in a
container is what matters. The main theorem gives a bound on $\mu(C)$, and
this readily supplies a bound on the number of edges $e(G[C])$ in the
container. If $G$ is regular, then this in turn leads to a bound on~$|C|$,
but there is no useful bound of this kind in general. What can be inferred,
however, is that each container $C$ can name some $v\in[n]$ for which
$|C\cap[v]|$ is bounded, and the number of containers naming any given~$v$
is small relative to~$v$. This arcane statement is expressed precisely in
\cite[Theorem~3.7]{ST2}; we don't restate the theorem here but point out
only that, as a consequence of the algorithm here, the theorem
holds for all, not just simple, graphs (and with the tuple replaced by a
single set).

We now state a somewhat weaker theorem than Theorem~\ref{thm:cover}, but
one which is just as good if the exact dependence of constants on the
parameter~$r$ is not an issue. The theorem has the twin advantages of being
easier to comprehend and being easier to prove. 

As mentioned before, the algorithm makes use of threshold
functions~$\theta_s$. The form of $\theta_s$ used to prove
Theorem~\ref{thm:cover} (given by Definition~\ref{defn:thetas}) is
carefully tuned to yield Theorem~\ref{thm:cover}~(d); this form of
$\theta_s$ in turn leads to the constraint $\delta(G,\tau)\le \zeta$ needed
to make the theorem hold.  For the weaker theorem we use more
straightforward functions $\theta_s$ (given by
Definition~\ref{defn:thetasw}). Likewise we need not be so careful about
the constraints on~$G$, resulting in a more transparent necessary
condition. Here is the weaker theorem.

\begin{thm}\label{thm:coverweak}
  Let $r\in\mathbb{N}$. Then there is a constant $c=c(r)>0$ such that the
  following holds.
  Let $G$ be an $r$-graph with average degree $d$ and vertex set~$[n]$. Let
  $0<\tau\le 1$ be chosen so that
  \begin{equation}
    d(\sigma)\le cd\tau^{|\sigma|-1}\quad \mbox{ holds for all $\sigma$,
      $|\sigma|\ge2$}\,.\tag{\dag}\label{eqn:dag}
  \end{equation}
  Then there is a function $C:\mathcal{P}[n]\to\mathcal{P}[n]$, such that,
  for every independent set $I\subset[n]$ there exists $T\subset I$ with
 \begin{itemize}
 \item[(a)] $I \subset C(T)$,
 \item[(b)] $\mu(T) \le \tau$, 
 \item[(c)] $|T|\le \tau n$,  and
 \item[(d)] $\mu(C(T)) \le 1 - c$.
 \end{itemize}
 Moreover
 $C(T)\cap[w]=C(T\cap[w])\cap[w]$ for all $T\in\mathcal{P}[n]$ and  $w\in[n]$.

 Indeed, the above holds for all sets $I\subset[n]$
 such that either $G[I]$ is $\lfloor c \tau^{r-1} d\rfloor$-degenerate
 or $e(G[I])\le c\tau^re(G)$.
\end{thm}

The necessary condition in Theorem~\ref{thm:coverweak} is stronger than
that in Theorem~\ref{thm:cover}, because it involves a bound on every
$d(\sigma)$, whereas the bound implicit in the function $\delta(G,\tau)$
involves some kind of averaging. This makes no difference, though, for the
applications presented here.

Most applications do not use the main theorem directly, but an iterated
form of it, resulting in much smaller containers.  Given an independent
set~$I$, the main theorem applied to $G$ supplies an initial container
$C_1$ for~$I$. A further application to $G[C_1]$ supplies a smaller
container~$C_2\subset C_1$, yet another application to $G[C_2]$
supplies~$C_3\subset C_2$, and so on, successive applications resulting in
a container $C$ that is very sparse.  This is how
\cite[Corollary~3.6]{ST2} was obtained, and we can derive the same
corollary from Theorem~\ref{thm:cover}, but with tuples replaced by
single sets. If the precise dependence of the constants on $r$ is not
important, though, Theorem~\ref{thm:coverweak} can be used multiple times
instead, which produces the
following corollary.

\begin{cor}\label{cor:sparse_container}
  Let $r\in\mathbb{N}$ and let $\epsilon>0$. Then there is a constant
  $c=c(r,\epsilon)$ for which the following holds. Let $G$ be an $r$-graph
  of average degree~$d$ on vertex set~$[n]$. Let $0<\tau\le1$ be
  chosen so that
  $$
    d(\sigma)\le cd\tau^{|\sigma|-1}\quad \mbox{ holds for all $\sigma$,
      $|\sigma|\ge2$}\,.
  $$
  Then there is a function $C:\mathcal{P}[n]\to\mathcal{P}[n]$, such that,
  for every independent set $I\subset[n]$ there exists $T\subset I$ with
 \begin{itemize}
 \item[(a)] $I \subset C(T)$,
 \item[(b)] $|T|\le \tau n$,  and
 \item[(c)] $e(G[C])\le \epsilon e(G)$.
 \end{itemize}
 Indeed, the above holds for all sets $I\subset[n]$ such that either $G[I]$
 is $\lfloor c \tau^{r-1} d\rfloor$-degenerate or $e(G[I])\le c\tau^re(G)$.
\end{cor}

We include a full proof of this corollary. It will be applied later to
prove Theorem~\ref{thm:eqn_cover}, from which nearly all our results on
linear equations will be derived. The only exception to this is a theorem
about the number of Sidon sets, for which we need a more technical version of
the corollary, Theorem~\ref{thm:iteration}. This version is very close to 
\cite[Theorem~6.3]{ST2} and so we shall not prove it.

\subsection{List colouring}\label{subsec:listcol}

Given an assignment $L:V(G)\to\mathcal{P}(\mathbb{N})$ of a list $L(v)$ of
colours to
each vertex~$v$, we say $G$ is $L$-chooseable if each vertex~$v$ can choose
a colour $f(v)\in L(v)$, such that there is no edge in which all the
vertices choose the same colour. The minimum number~$k$ such that $G$ is
$L$-choosable whenever $|L(v)|\ge k$ for every~$v$ is called the {\em
  list-chromatic} number of~$G$, denoted by $\chi_l(G)$. This notion was
introduced for graphs by Vizing~\cite{V} and by Erd\H{o}s, Rubin and
Taylor~\cite{ERT}. It was studied for Steiner systems by Haxell and
Pei~\cite{HP}, for simple regular 3-graphs by Haxell and
Verstra\"ete~\cite{HV} and for certain $r$-graphs by Alon and
Kostochka~\cite{AK1,AK2}. The case of the next theorem for simple graphs
was stated as~\cite[Theorem~2.1]{ST2}; it strengthens and extends a theorem
of Alon~\cite{A2} for 2-graphs. It is possible to obtain a lower bound on
$\chi_l(G)$ for a non-simple hypergraph by applying~\cite[Theorem~2.1]{ST2}
to a randomly chosen simple subgraph, but the theorem here gives a better
bound.

\begin{thm}\label{cor:chil}
  Let $r\in\mathbb N$ be fixed. Let $G$ be an $r$-graph with average
  degree~$d$. Suppose that $d^{(j)}(v)\le d^{(r-j)/(r-1)+o(1)}$ for every
  $v\in V(G)$ and for $2\le j\le r$, where $o(1)\to0$ as $d\to\infty$. Then
  $$
  \chi_l(G)\,\ge\,\,(1+o(1))\,\frac{1}{(r-1)^2} \log_rd \,.
  $$
  Moreover, if $G$ is regular then
  $$
  \chi_l(G)\,\ge\,\,(1+o(1))\,\frac{1}{r-1}\log_rd\,.
  $$
\end{thm}

The proof of \cite[Theorem~2.1]{ST2} illustrates how to prove theorems of
this kind and so we do not prove Theorem~\ref{cor:chil} here. We remark
only that it uses~\cite[Theorem~3.7]{ST2}, which can now be applied to all
$r$-graphs and not just simple ones, due to the online property mentioned
above in~\S\ref{subsec:containers}. For readers keen to check the details,
we note that, given the choice of $\zeta$, $\tau$ and $k$ in the proof, the
condition $d^{(j)}(v)\le d^{(r-j)/(r-1)}\zeta^{-1}$ implies that
$\delta(G,\tau)\le \zeta$ as needed, and that $d^{(2)}(v)\le
d^{(r-2)/(r-1)}\zeta^{-1}$ implies $d(v)\le nd^{(r-2)/(r-1)}\zeta^{-1}$
from which $\mu([k])\le
(1/nd)knd^{(r-2)/(r-1)}\zeta^{-1}\le\zeta^5\le\zeta/2r!$ follows. The rest
of the proof is identical.

\subsection{Linear equations}\label{subsec:lineq}

A subset $S\subset[n]$ is said to be {\em sum-free} if there is no solution
to $x+y=z$ with $x,y,z\in[n]$. Cameron and Erd\H{o}s conjectured that the
number of sum-free sets is $O(2^{n/2})$, and this conjecture is prominent
in the history of the container method. Alon~\cite{A3}, Calkin~\cite{C},
and Erd\H{o}s and Granville (unpublished) each proved there are at most
$2^{n/2+o(n)}$ sum-free sets. Green~\cite{G0} and Sapozhenko~\cite{Sap3}
proved the conjecture. Sapozhenko's proof makes use of containers for
2-graphs. Green's argument is of interest here because he highlighted what
was, in effect, the usefulness of containers for 3-uniform hypergraphs. The
argument in~\cite{G0} is explicitly split into two parts. In the first
part, a collection $\mathcal{C}$ of sets is found such that each sum-free
set is a subset of a member of $\mathcal{C}$, each $C\in\mathcal{C}$ has
very few
solutions $x+y=z$ with $x,y,z\in C$, and $\mathcal{C}$ is small,
specifically, $|\mathcal{C}|=2^{o(n)}$. In the present terminology
$\mathcal{C}$ is simply a set of containers for the 3-uniform hypergraph
whose edges are the triples $\{x,y,z\}$ with $x+y=z$. (Indeed we have
changed Green's notation $\mathcal{F}$ to our $\mathcal{C}$ for
consistency.)
In the second part of
the argument, a detailed inspection of the containers leads to a proof of
the conjecture.

Green produced his containers by the granularization technique developed by
Green and Ruzsa~\cite{GR}. The containers are small perturbations of unions
of arithmetic progressions, and a container is found for each sum-free
set by means of Fourier techniques. The method perhaps gives more
containers than what is given by the method of~\cite{ST3} or the present
method, but that is irrelevant in the context.

Here we are interested in more general systems of linear equations, of the
form $A x = b$, where $A$ is a $k \times r$ matrix with entries in $F$, $x
\in F^r$ and $b \in F^k$; $F$ itself might be either a finite field or the
set of integers $[N]$. We include also the possibility that $F$ is an
abelian group: in this case, $A$ should have integer entries, where
integer-group multiplication $a x$, $a \in \mathbb{Z}$, $x \in F$, is $a$
copies of $x$, $x + \cdots + x$; or $-a$ copies of $-x$ if $a$ is
negative. We call a triple $(F,A,b)$ of this kind a $k\times r$ {\em linear
  system}.

\begin{defn}\label{defn:exteqnoX}
For a $k\times r$ linear system $(F,A,b)$, a subset $I \subset F$ is
\emph{solution-free} if there is no $x \in I^r$ with $Ax=b$.
\end{defn}

It takes little imagination to think that a container theorem for
solution-free sets might come in handy. Such a theorem can be obtained if
we write down some $r$-graph $G$ whose edges represent solutions and whose
independent sets represent sum-free sets. Then
Corollary~\ref{cor:sparse_container} will supply containers; all we need do
is to compute the subset degrees~$d(\sigma)$. These degrees, and hence the
number of containers, depend on a parameter $m_F(A)$ discussed by R\"odl
and Ruci\'nski~\cite{RR}; we defer the details
(Definition~\ref{def:mA}). The definition requires $A$ to satisfy a mild
condition, but a necessary one (see~\S\ref{sec:sumfree}).

\begin{defn}
 We say that $A$ has \emph{full rank} if given any $b \in
F^k$ there exists $x \in F^r$ with $Ax=b$. We then say that $A$ is
\emph{abundant} if it has full rank and every $k \times (r-2)$ submatrix
obtained by removing a pair of columns from $A$ still has full rank.
\end{defn}

(This definition of full rank might be non-standard.) The resultant
container theorem, for systems with abundant matrices, is
Theorem~\ref{thm:eqn_cover}. We illustrate it with two applications, one to
the number of solution-free subsets and the other to the size of
solution-free subsets in sparse randomly chosen sets. The first requires
only that the number of containers be $2^{o(n)}$, but the second needs a
much smaller collection.

The containers produced by Theorem~\ref{thm:eqn_cover} are not
solution-free but are nearly so. For the applications, it is necessary that
this property implies the containers $C$ are not much larger than maximum
solution free sets. Roughly speaking, we would like to say if $C^r$
contains $o(|F|^{r-k})$ solutions to the system $(F,A,b)$ then
$|C|\le\mbox{ex}(F,A,b)+o(|F|)$, where $\mbox{ex}(F,A,b)$ is the maximum
size of a solution-free subset for the linear system $(F,A,b)$. To make this
precise we use the following definition. A {\em null} function
$f:\mathbb{R}^+\to\mathbb{R}^+$ is one such that $f(x)\to0$ as $x\to0$.

\begin{defn}\label{defn:supersat}
  Let $f:\mathbb{R}^+\to\mathbb{R}^+$ be null. The $k\times r$ linear
  system $(F,A,b)$ is said to be $f$-{\em supersaturated} if, whenever
  $X\subset F$ contains at most $\eta|F|^{r-k}$ solutions to $Ax=b$, then
  $|X|\le \mbox{ex}(F,A,b) + f(\eta)|F|$, where $\mbox{ex}(F,A,b)$ is the
  maximum size of a solution-free subset for $(F,A,b)$.
\end{defn}

Obviously every system, being finite, is $f$-supersaturated for some
null~$f$, so the definition has content only when $f$ does not depend on
$(F,A,b)$ directly.  Systems of interest are often supersaturated but the
verification generally requires a removal lemma. More is said about this
in~\S\ref{subsec:supersat}. Our first application of containers to linear
systems is to estimate the number of solution-free sets, in a similar way
to the weak versions of the Cameron-Erd\H{o}s conjecture; if $(F,A,b)$ is
supersaturated and $A$ is abundant, then there are
$2^{{\rm ex}(F,A,b)+o(|F|)}$ solution-free sets. There are several results
of this nature in the literature: we mention only that
Sapozhenko~\cite{Sap6} obtained one by a container argument applied with a
theorem of Lev, \L{u}czak and Schoen~\cite{LLS} (see~\cite{Sap5}).

\begin{thm}\label{thm:eqn_count}
  Let $k,r\in\mathbb{N}$, let $f:\mathbb{R}^+\to\mathbb{R}^+$ be null and
  let $\epsilon>0$. Then there exists $c=c(k,r,f,\epsilon)$ (or
  $c=c(A,f,\epsilon)$ in the case $F=[N]$) such that, if $(F,A,b)$ is a $k
  \times r$ $f$-supersaturated linear system with $|F|>c$, and~$A$ is
  abundant, then the number of solution-free subsets of $F$ is
  $2^{{\rm ex}(F,A,b)+\lambda|F|}$, where $0\le\lambda<\epsilon$.
\end{thm}

Our second application is to the size of solution-free subsets within
randomly chosen subsets $X\subset F$. Let the elements of $X$ be chosen
independently at random with probability~$p$. Clearly one might expect to
find a solution-free subset of size at least $p\,{\rm ex}(F,A,b)$ within~$X$,
and it turns out that if $p$ is not too small then this is the largest that
a solution-free subset of $X$ can be. In a proof of this by the container
method, the size of $\mathcal{C}$ is the factor that determines how small a
$p$ the proof holds for. The full statement appears in
Theorem~\ref{thm:eqn_sparse}, and it requires that $(F,A,b)$ has the
supersaturation property. Nevertheless we state a special case here,
because in this case supersaturation is easily verified by a simple density
argument (no removal lemma is needed). For $\ell=3$ the theorem was proved
by Kohayakawa, \L{u}czak and R\"odl~\cite{KLRR}.

\begin{thm}[Conlon and Gowers~\cite{CG},
  Schacht~\cite{Sch}]\label{thm:szem_sparse}
Let $\ell\ge3$ and $\epsilon>0$. There exists a constant $c>0$ such that
for $p\ge cN^{-1/(\ell-1)}$, if $X \subset [N]$ is a random
subset chosen with probability $p$, then with probability tending to $1$ as
$N \to \infty$, any subset of $X$ of size $\epsilon|X|$ contains an arithmetic
progression of length $\ell$.
\end{thm}

Note that the bound on $p$ here is best possible (up to the value of~$c$),
as indicated by the fact that if $p=o(N^{-1/(\ell-1)})$ then $X$ contains
(in expectation) many fewer than $|X|$ arithmetic progressions and hence
contains a large subset free of them.

As well as these two applications of Theorem~\ref{thm:eqn_cover}, we prove
a bound on the number of Sidon sets, which are sets $S\subset[n]$ for which
every sum of two elements is distinct, i.e., there are no solutions to
$w+x=y+z$ with $\{w,x\}\ne\{y,z\}$. Erd\H{o}s and Tur\'an~\cite{ET} proved
that $|S|\le (1+o(1))\sqrt{n}$, and Cameron and Erd\H{o}s~\cite{CE} raised
the question of how many Sidon sets there are.

\begin{thm}\label{thm:sidon}
There are between $2^{(1.16+o(1))\sqrt{n}}$ and $2^{(55+o(1))\sqrt{n}}$
Sidon subsets of~$[n]$.
\end{thm}

In particular there are more than $2^{(1+o(1))\sqrt{n}}$ Sidon sets. The
upper bound comes from a direct application of a standard container
argument: Kohayakawa, Lee, R\"odl and Samotij~\cite{KLRS} have a finer
argument (with a better constant).

\section{The algorithm}

We remark at the outset that the construction nowhere makes use
of the fact that $I$ is independent. Indeed we shall take advantage of this
fact to build containers for sparse sets. The independence, or sparsity,
comes into play only later, in the calculation of the number of containers
required and of their sizes.

As in~\cite[Section~4]{ST2}, the process for constructing containers can be
described in terms of an algorithm with two slightly different modes,
``prune'' mode and ``build'' mode. In prune mode, the algorithm takes as
input a set~$I$ and outputs a subset $T\subset I$. In build mode, the
algorithm takes as input some set~$T$ and outputs a set~$C$. The set $C$ is
thus a function of $T$ and we can emphasise this by writing $C=C(T)$. The
two modes of the algorithm should co-operate in the following sense, that
if the set $T$ input to build mode is the one output by prune mode with
input~$I$, then $I\subset C(T)$ must hold.

Both modes of the algorithm have available the hypergraph $G$ together with
an enumeration, or labelling, of the vertex set, which we take to
be~$[n]$. As far as the algorithm is concerned, there is nothing special
about the enumeration; changing the enumeration might change which sets
actually become containers, but the properties of them, as described in
Theorem~\ref{thm:cover}, remain the same. This comment covers the online
property also. (Note, however, that the only application to date of the
online property is~\cite[Theorem~3.7]{ST2}, described
in~\S\ref{subsec:containers}; this application uses the containers produced
when the enumeration is by order of decreasing degree.)

\subsection{General properties of the algorithm}\label{subsec:goldenrule}

Prune mode initialises $T=\emptyset$ and build mode initialises
$C=[n]$. Both modes of the algorithm then run through the vertices one by
one in order, and check, for each vertex $v$, whether $v$ satisfies a
membership rule; the same rule must be used in each mode of the algorithm.
If the rule is not satisfied, neither mode takes any action and the
algorithm moves on to the next vertex. If, however, the rule is satisfied
then prune mode adds $v$ to $T$ if $v\in I$, whereas build mode removes $v$
from $C$ if $v\notin T$.

It can be seen that, whatever rule is used, prune mode outputs the set $T$
of members of $I$ that satisfy the rule, and build mode outputs the set $C$
comprising $T$ together with all vertices not satisfying the rule. So
plainly, if $T$ is the set output by prune mode with input~$I$, then
$I\subset C(T)$, as required.

Notice that, at the point when the vertex~$v$ is inspected, both modes of
the algorithm know the set $T\cap [v-1]$, this being the members of $T$
that lie in the range $1,\ldots,v-1$ of vertices that have been examined so
far.  It is permissible, therefore, for the membership rule to depend on
this set, because both modes of the algorithm will be able to evaluate the
rule in the same way. We shall express this dependence as follows. There
will be some data structure $\mathcal{D}$ (say, a collection of sets or
hypergraphs), that is initialised at the start of the algorithm and which is
updated whenever a vertex $v$ is added to~$T$. The rule can then depend
on~$\mathcal{D}$. To be more precise, when $v$ is inspected, $\mathcal{D}$
is a function of $T\cap [v-1]$, and the membership rule is a function of
$\mathcal{D}$ and of $v$. We remind the reader here that complete knowledge
of the hypergraph $G$ is available throughout the procedure.

The general form of such an algorithm is set out in
Table~\ref{table:generalg}.
\begin{table}
\begin{center}
\begin{tabbing}
\quad\quad\=\\
\>{\sc Input}\\
\>\qquad\= $r$-graph $G$ on vertex set~$[n]$\\
\>\>{\em in prune mode} \=a subset $I\subset [n]$\\
\>\>{\em in build mode} \>a subset $T\subset[n]$\\
\>{\sc Output}\\
\>\>{\em in prune mode} \> a subset $T\subset [n]$\\
\>\>{\em in build mode} \>a subset $C\subset [n]$\\
\>{\sc Initialisation} \\
\>\>initialise data structure $\mathcal{D}$\\
\>\>{\em in prune mode} \= put $T=\emptyset$\\
\>\>{\em in build mode} \>put $C=[n]$\\
\\
\>for $v=1,2,\ldots,n$ do:\\
\>\quad\=let Rule$(v)$ be some condition on $v$ depending on $\mathcal{D}$\\
\>\>if \=Rule$(v)$ is satisfied\\
\>\>\>{\em in prune mode} \= if $v\in I$, add $v$ to $T$ \\
\>\>\>{\em in build mode} \>if $v\notin T$, remove $v$ from $C$\\
\>\>\>if \=$v\in T$\\
\>\>\>\>update $\mathcal{D}$ using $v$
\\
\end{tabbing}
\end{center}
\caption{\label{table:generalg}The general form of the online algorithm}
\end{table}
Note that $\mathcal{D}$ is updated with $v$ only if the rule is passed and
$v\in T$; in the case of prune mode, this is after $v$ has been added
to $T$.

\subsection{Comments on the general form}

Before describing the particular membership rule which will be used in the
actual container algorithm, we make one or two observations that apply
generally.

\begin{rem}{\em The online property.}\label{rem:online}
  It is immediate that an algorithm of the kind described will produce
  containers with the online property: that is, given any initial ordering
  of the vertices, then for all $w\in[n]$, $C(T)\cap[w]=C(T\cap[w])\cap[w]$
  holds. This is simply because the algorithm has already determined
  $C(T)\cap[w]$ by the time it has inspected $v=1,2,\ldots,w$, and for this
  range of $v$ the decisions made depend only on which elements of $[w]$
  are contained in the input, that is, on $T\cap[w]$.
\end{rem}
\begin{rem}{\em Over-specifying the input to build mode.}\label{rem:enlarge}
  It will be helpful, in order to make the presentation cleaner, to observe
  that if $I$ is a set for which prune mode outputs~$T$, then $C(S)=C(T)$
  for any set $S$ such that $T\subset S\subset I$. This is because $T$ is
  the subset of $I$ for which the rule is satisfied: hence for any $v\in
  S\setminus T$ the rule is not satisfied, and both forms of the algorithm
  pass over $v$ without further action. (In particular, $\mathcal{D}$ is
  not updated when $v\in S\setminus T$.) This observation was made by
  Balogh, Morris and Samotij~\cite{BMS} in their algorithm.
  
  The advantage of over-specifying comes when the container theorem is
  iterated, as mentioned regarding the proof of
  Corollary~\ref{cor:sparse_container}. Each of the containers
  $C_1,C_2,C_3,\ldots$ is determined by some subset $T_1,T_2,T_3,\ldots$
  of~$I$. We could instead give as input to build mode the set $T=T_1\cup
  T_2\cup T_3\cup\cdots$ at every iteration of the algorithm, and the sets
  $C_1,C_2,C_3,\ldots$ would be output correctly. Hence the final
  container~$C$ is determined by~$T$.
\end{rem}
\begin{rem}{\em Changing the vertex order.}\label{rem:changeorder}
  Rather than process the vertices one by one in the order supplied, that
  is, rather than test the rule in the order $v=1,2,\ldots,n$, the
  algorithm could be changed so that it decides for itself in which order
  to process the vertices. For example, having processed vertices
  $v=v_1,v_2,\ldots,v_{k-1}$, it could inspect each of the remaining
  vertices and choose as $v_k$ the one which maximises some potential
  function that depends on~$\mathcal{D}$. A feature like this is used
  in~\cite{BMS}. Note that the algorithm here will
  still work under this dynamic re-ordering, which is to say the fact that
  $I\subset C(T)$ is preserved.
  However it will break the online property, because
  $C\cap[w]$ no longer depends just on $T\cap[w]$. For this reason, and
  because the rest of our analysis appears to gain no advantage from it, we
  do not use this feature.
\end{rem}

\subsection{The actual container algorithm}\label{subsec:actual}

The main features of the algorithm are described in~\cite[Section~4]{ST2},
so we do not dwell on them here. We make use of auxiliary multisets of
edges $P_{r},P_{r-1},\ldots,P_1$. Here $P_s$ is a multiset of $s$-sets in
$[n]$; in~\cite{ST2} $P_s$ was an $s$-uniform multi-hypergraph but here it
is more convenient to use the same symbol for the edge set of such a
hypergraph, the difference being purely one of notation. We take
$P_r=E(G)$, but for $s<r$ the multisets $P_s$ will grow during the course
of the algorithm, being initialised to the empty set. Any set added to
$P_s$ comes from a set in $P_{s+1}$ after removing its first (in the
ordering of~$[n])$ vertex~$v$; if this happens, it does so when it is $v$'s
turn to be processed by the algorithm, and then only if $v\in T$. It
follows that each set $f\in P_s$ comes from some set $t\subset T$,
$|t|=r-s$, such that $t\cup f \in E(G)$, and $f$ comprises the last $s$
vertices of the edge $t\cup f$. For each set $\sigma\in[n]^{(\le s)}$, we
denote by $d_s(\sigma)$ the degree of $\sigma$ in $P_s$, that is,
$d_s(\sigma)$ is the number of edges in $P_s$ that contain~$\sigma$. Thus
the value of $d_s(\sigma)$ is initially zero, and it grows during the run
of the algorithm as $P_s$ grows.  Finally, $\Gamma_s$ is the set of
elements $\sigma\in[n]^{(\le s)}$ whose degree $d_s(\sigma)$ has reached
some predetermined threshold. The data structure $\mathcal{D}$, on which
the membership rule is based, comprises $P_{r},P_{r-1},\ldots,P_1$ together
with $\Gamma_{r-1},\Gamma_{r-2},\ldots,\Gamma_1$.

The algorithm of~\cite{ST2} involves two parameters, $\tau$ and $\zeta$,
and these are used here in the same way; roughly speaking, $\tau$ is the
measure of~$T$ and $\zeta$ is a smallish constant. In particular, if $d$ is
the average degree of $G$ then we denote by $B$ the set of vertices of low
degree, that is, $B=\{v\in[n]:d(v)<\zeta d\}$. 

The final thing needed for the algorithm is the threshold function used to
determine entry into $\Gamma_s$. As has been said before, we make use of
two different threshold functions; we call these a {\em strong} function
for the proof of Theorem~\ref{thm:cover} and a {\em weak} function for the
proof of Theorem~\ref{thm:coverweak}. Both functions depend on the input
parameter $\tau$ and on the graph $G$ itself. But it is important to note
that their values can be computed at the start of the algorithm and they do
not change during running.

\begin{defn}\label{defn:thetas}
For $s=2,\ldots,r$ and $\sigma \in [n]^{(\le s)}$, the {\em strong}
threshold functions $\theta_s$ are given as follows.
\begin{align*}
\theta_s(\sigma)&=\tau^{r-s}d(v)&\mbox{for $\sigma=\{v\}$, i.e.\ $|\sigma|=1$}\\
\theta_s(\sigma)&= 2^{r\choose2}\tau^{r-s}\sum_{\ell=0}^{r-s} 
2^{-{s+\ell\choose2}}\tau^{-\ell}d^{(|\sigma|+\ell)}(\sigma)&\mbox{for $|\sigma|\ge 2$}\\
\end{align*}
\end{defn}

The definition of the weak threshold function makes use of a real
number~$\delta$. This is not the co-degree function $\delta(G,\tau)$, which
plays no part in Theorem~\ref{thm:coverweak}, but it is a kind of weak echo
of this function, and so we use the same letter.

\begin{defn}\label{defn:thetasw}
  For $s=2,\ldots,r$ and $\sigma \in [n]^{(\le s)}$, the {\em weak}
  threshold functions $\theta_s$ are given as follows, where $\delta$ is
  the minimum real number such that $d(\sigma)\le \delta
  d\tau^{|\sigma|-1}$ holds for all $\sigma$, $|\sigma|\ge2$.
\begin{align*}
\theta_s(\sigma)&=\tau^{r-s}d(v)&\mbox{for $\sigma=\{v\}$, i.e.\ $|\sigma|=1$}\\
\theta_s(\sigma)&= \delta d\tau^{r-s+|\sigma|-1}&\mbox{for $|\sigma|\ge 2$}\\
\end{align*}
\end{defn}

The way the threshold functions are used in the algorithm means that the
degrees $d_s(v)$ are bounded, as shown in Lemma~\ref{lem:maxdeg} for the
strong functions and in Lemma~\ref{lem:maxdegw} for the weak
versions. These lemmas form the foundation of the proofs of the container
theorems. The strong versions have an extra consequence for
$d_{s-1}(\sigma)$, stated in Lemma~\ref{lem:sigmadeg}, which is
needed to establish the bound on $\mu(C)$ in Theorem~\ref{thm:cover}~(d).

The container algorithm is set out in Table~\ref{table:actualg}.
\begin{table}
\begin{center}
\begin{tabbing}
\quad\quad\=\\
\>{\sc Input}\\
\> \qquad\=an $r$-graph $G$ on vertex set~$[n]$, with average
degree $d$\\
\>\>parameters $\tau, \zeta>0$\\
\>\>{\em in prune mode} \=a subset $I\subset [n]$\\
\>\>{\em in build mode} \>a subset $T\subset[n]$\\
\>{\sc Output}\\
\>\>{\em in prune mode} \> a subset $T\subset [n]$\\
\>\>{\em in build mode} \>a subset $C\subset [n]$\\
\>{\sc Initialisation}\\
\>\> put $B=\{v\in[n]: d(v) < \zeta d\}$\\
\>\>evaluate the thresholds $\theta_s(\sigma)$, $\sigma\in[n]^{(\le s)}$,
$1\le i\le r$\\
\>\>put $P_r=E(G)$, $P_s=\emptyset$, $\Gamma_s=\emptyset$,
$s=1,2,\ldots,r-1$\\
\>\>{\em in prune mode} \= put $T=\emptyset$\\
\>\>{\em in build mode} \>put $C=[n]$\\
\\
\>for $v=1,2,\ldots,n$ do:\\
\>\quad\=for $s=1,2,\ldots,r-1$ do:\\
\>\>\quad\=let $F_{v,s} = \{f \in [v+1,n]^{(s)} : \{v\}\cup f\in P_{s+1},\,
\mbox{ and }\not\exists\, \sigma\in \Gamma_s\,, \sigma\subset f\,\}$\\
\>\>\quad[{\em here $F_{v,s}$ is a multiset with multiplicities inherited from 
$P_{s+1}$}]\\
\>\>if \=$v\notin B$, and either $|F_{v,s}|\ge\zeta
\tau^{r-s-1}d(v)$ for some $s$ or $v\in\Gamma_1$\\
\>\>\>{\em in prune mode} \= if $v\in I$, add $v$ to $T$ \\
\>\>\>{\em in build mode} \>if $v\notin T$, remove $v$ from $C$\\
\>\>\>if \=$v\in T$ then for $s=1,2,\ldots,r-1$ do:\\
\>\>\>\>add $F_{v,s}$ to $P_s$\\
\>\>\>\>for each $\sigma\in[v+1,n]^{(\le s)}$, if $d_s(\sigma)\ge
\theta_s(\sigma)$, add $\sigma$ to $\Gamma_s$\\
\end{tabbing}
\end{center}
\caption{\label{table:actualg}The container algorithm}
\end{table}
The membership rule test is the line that begins ``if $v\notin B$
\dots''. The two lines before that are merely to define the multisets
$F_{v,1},\ldots,F_{v,r-1}$ that are used in the test.

As in~\cite{ST2}, the aim is to build up the multisets $P_s$ as quickly as
possible, whilst keeping the degrees in $P_s$ of each set $\sigma$ below
its target value $\theta_s(\sigma)$. The set $\Gamma_s$ comprises those
$\sigma$ that have reached their target degree in~$P_s$. Hence the multiset
$F_{v,s}$ is the potential contribution of $v$ to $P_s$; it is the edges of
$P_{s+1}$ that contain~$v$ (with $v$ then removed), but which don't contain
anything from $\Gamma_s$. If $F_{v,s}$ is large for some $s$ then $v$ makes
a substantial contribution to that $P_s$, and we place $v$ in~$T$, updating
all $P_s$ and $\Gamma_s$ accordingly.

If every vertex that enters $T$ does so because one of the sets $F_{v,s}$ is
substantial, then $T$ will be small, because the size of each $P_s$ is
bounded (this is why we cap its degrees) and so it cannot be increased
often. Observe, though, that there is another reason for placing $v$
in~$T$, other than that one of the $F_{v,s}$ is large, namely, that
$v\in\Gamma_1$. This will never happen if~$I$ is an independent set, since
it means that $\{v\}$ is an edge of $P_1$, which, as mentioned earlier,
means $t\cup\{v\}$ is an edge of~$I$ for some $t\in T^{(r-1)}\subset
I^{(r-1)}$. If, however, $I$ is not independent then some vertices might
enter $T$ for this reason, but provided $I$ is sparse this will be a rare
occurrence and $T$ will still be small, as required.

\subsection{Differences from previous algorithm}\label{subsec:diffs}

The container algorithm used in~\cite{ST2} was similar to the one here,
except that the sets $P_{r-1},\ldots,P_1$ were built consecutively by $r-1$
passes of the algorithm, rather than in parallel during one pass as
here. In other words, $P_s$ was constructed after $P_{s+1}$ had been fully
built. The construction of $P_s$ produced a set $T_s$ (in prune mode) and
so the whole algorithm produced a tuple $(T_{r-1},\ldots,T_1,T_0)$ rather
than a single set~$T$. However this is not the essential difference
between the algorithms, since for the reason given in
Remark~\ref{rem:enlarge} the tuple could have been replaced by the set
$T_{r-1}\cup\cdots\cup T_1\cup T_0$. (The set $T_0$ was defined to be
$I\cap \Gamma_1$, which is incorporated into the set $T$ in the present
algorithm by means of the membership rule.)

The main difference between this and the previous algorithm is the
condition for entry into $\Gamma_s$. In~\cite{ST2} the condition depended
on knowledge of the whole of $P_s$. Here, it depends on $\theta_s(\sigma)$,
which is available from the start of the algorithm. It is this that allows
the sets $P_s$ to be computed in parallel. These sets will not be exactly
the same as those in~\cite{ST2} due to the difference in detail, but they
will be similar.

\section{Analysis of the algorithm}\label{sec:analysis}

We now analyse the behaviour of the algorithm when using the strong
threshold functions and when using the weak threshold functions, and so
establish the bounds on $\mu(T)$ and $\mu(C)$ claimed in
Theorems~\ref{thm:cover} and~\ref{thm:coverweak}. In both cases, we need to
obtain bounds on the degrees $d_s(u)$ of vertices. In the case of the
strong form, which we begin with, the bounds obtained are the same as those
in~\cite{ST2}. Fortunately this means we can then make direct use of the
arguments used in~\cite{ST2} to bound $\mu(T)$ and $\mu(C)$, without
repeating the details. In the case of the weak form, we derive
corresponding bounds on the degrees $d_s(u)$, and then give proofs of
bounds on $\mu(T)$ and $\mu(C)$ which follow from these. Note that, in the
statements of the lemmas, $d_s(u)$ and $d_s(\sigma)$ refers to the final
degrees in $P_s$, after the algorithm has completed.

The lemmas in this section all make claims about the output of the
algorithm, given certain inputs. It is worth emphasising that the values of
$\tau$ and $\zeta$ that appear in the lemmas are {\em the values of the
  parameters which are input to the algorithm}. In particular, we
re-iterate the remark made just prior to Definition~\ref{defn:thetas}, that
the threshold functions are evaluated during the initialisation phase of
the algorithm, and the definitions of these functions is in terms of the
parameter $\tau$ that is input to the algorithm. Consequently, the value of
$\delta(G,\tau)$, mentioned in Lemma~\ref{lem:maxdeg} where the strong
threshold functions are being used, and the value of~$\delta$, appearing in
the lemmas of~\S\ref{subsec:weak} where the weak threshold functions are in
use, are determined by the value of~$\tau$ input to the algorithm. (The
value of $\delta(G,\tau)$ is defined prior to Theorem~\ref{thm:cover}, and
the value of $\delta$ is defined by Definition~\ref{defn:thetasw}.)

\subsection{Strong thresholds and the proof of 
Theorem~\ref{thm:cover}}\label{subsec:strongth}

The following fundamental lemma gives bounds on subset degrees in $P_s$.

\begin{lem}\label{lem:sigmadeg} Let the algorithm be run using the strong
  threshold functions. Then
for $s=2,\ldots,r$ and $2\le|\sigma|\le s$, we have
$$
d_s(\sigma)\le 
2^{r\choose2}\tau^{r-s}\sum_{\ell=0}^{r-s} 
2^{-{s+\ell\choose2}+\ell}\tau^{-\ell}d^{(|\sigma|+\ell)}(\sigma)
$$
and for $2\le|\sigma|\le s-1$ with $\sigma\in\Gamma_{s-1}$, we have
$$
d_{s-1}(\sigma)\ge 2^{s-1}\tau d_s(\sigma)\,.
$$
\end{lem}
\begin{rem}\label{rem:sigmadeg}
  During the algorithm, $\sigma$ is placed into $\Gamma_s$ as soon as
  $d_s(\sigma)$ is at least $\theta_s(\sigma)$, and the degree will not
  thereafter increase. The degree can be greater than $\theta_s(\sigma)$
  but the first inequality of the lemma shows that the excess over
  $\theta_s(\sigma)$ is not large. The second inequality, on the other
  hand, is one which was built into the original algorithm. The lemma shows
  that the inequality is valid in the present setup too, enabling us to
  copy over earlier proofs without change. The definition of
  $\theta_s(\sigma)$ was made with this in mind.
\end{rem}
\begin{proof}
  We prove the first inequality for $s=r,r-1,\ldots,2$ in order. For $s=r$
  the inequality asserts that $d_r(\sigma)\le d^{(|\sigma|)}(\sigma)$,
  which is true. Suppose then that the corresponding inequality for
  $d_{s+1}()$ holds.  If $\sigma\in \Gamma_s$ then $\sigma$ entered
  $\Gamma_s$ after some vertex $v$ was inspected and the set $F_{v,s}$ was
  added to $P_s$. Before this addition, $d_s(\sigma)\le \theta_s(\sigma)$
  was true. The increase in $d_s(\sigma)$ resulting from the addition is
  the number of $s$-sets in $F_{v,s}$ that contain $\sigma$. By definition of
  $F_{v,s}$, these come from edges of $P_{s+1}$ that contain both $v$
  and~$\sigma$; the number of these is at most $d_{s+1}(\{v\}\cup\sigma)$.
  The value of $d_s(\sigma)$ remains unchanged after the addition, and so
  at the end we have
  $d_s(\sigma)\le\theta_s(\sigma)+d_{s+1}(\{v\}\cup\sigma)$ (for some $v$
  depending on $\sigma$). This inequality trivially holds if
  $\sigma\notin\Gamma_s$, and so it holds for all $\sigma\in[n]^{(\le
    s)}$. The induction hypothesis supplies an upper bound for $d_{s+1}()$,
  and so, bearing in mind that $d^{(j)}(\{v\}\cup\sigma)\le
  d^{(j)}(\sigma)$ for all $j$ by definition, we obtain
  \begin{align*}
    d_s(\sigma)&\le\theta_s(\sigma)+d_{s+1}(\{v\}\cup\sigma)\\
    &\le \theta_s(\sigma)+2^{r\choose2}\tau^{r-s-1}\sum_{\ell=0}^{r-s-1}
    2^{-{s+1+\ell\choose2}+\ell}\tau^{-\ell}d^{(|\sigma|+1+\ell)}(\{v\}\cup\sigma)\\
    &\le \theta_s(\sigma)+2^{r\choose2}\tau^{r-s-1}\sum_{\ell=0}^{r-s-1}
    2^{-{s+1+\ell\choose2}+\ell}\tau^{-\ell}d^{(|\sigma|+1+\ell)}(\sigma)\displaybreak[0]\\
    &=\theta_s(\sigma)+2^{r\choose2}\tau^{r-s}\sum_{\ell=1}^{r-s}
    2^{-{s+\ell\choose2}+\ell-1}\tau^{-\ell}d^{(|\sigma|+\ell)}(\sigma)\\
    &=2^{r\choose2}\tau^{r-s}\left(2^{-{s\choose2}}d(\sigma)+ \sum_{\ell=1}^{r-s}
    2^{-{s+\ell\choose2}}\left[1+2^{\ell-1}\right]
    \tau^{-\ell}d^{(|\sigma|+\ell)}(\sigma)\right)\\
    &\le2^{r\choose2}\tau^{r-s}\sum_{\ell=0}^{r-s}
    2^{-{s+\ell\choose2}+\ell}\tau^{-\ell}d^{(|\sigma|+\ell)}(\sigma)\,,
  \end{align*}
  since $1+2^{\ell-1}\le 2^\ell$ for $\ell\ge1$. This finishes the
  verification of the first inequality of the lemma.

  To prove the second inequality, note that if $\sigma\in\Gamma_{s-1}$ then
  $d_{s-1}(\sigma)\ge\theta_{s-1}(\sigma)$, and so, using the first
  inequality, it is enough to show that
  $$
  \theta_{s-1}(\sigma)\ge2^{s-1}\tau2^{r\choose2}\tau^{r-s}\sum_{\ell=0}^{r-s}
  2^{-{s+\ell\choose2}+\ell}\tau^{-\ell}d^{(|\sigma|+\ell)}(\sigma)
  $$
  holds, i.e., that
  $$
  \sum_{\ell=0}^{r-s+1}
  2^{-{s-1+\ell\choose2}}\tau^{-\ell}d^{(|\sigma|+\ell)}(\sigma)
  \ge2^{s-1}\sum_{\ell=0}^{r-s}
  2^{-{s+\ell\choose2}+\ell}\tau^{-\ell}d^{(|\sigma|+\ell)}(\sigma)
  $$
  holds. For this, it suffices that $-{s-1+\ell\choose2}\ge
  s-1-{s+\ell\choose2}+\ell$ for $0\le\ell\le r-s$. But this inequality
  holds identically, and the proof of the lemma is complete.
\end{proof}

Using Lemma~\ref{lem:sigmadeg} we can immediately prove the next lemma,
which is identical to~\cite[Lemma~5.2]{ST2} (apart from trivial changes of
wording due to the different algorithm here).

\begin{lem}\label{lem:maxdeg}
  Let $G$ be an $r$-graph on vertex set $[n]$ with average degree~$d$. Let
  $P_r=E(G)$ and let $P_{r-1},\ldots,P_1$ be the multisets constructed
  during the algorithm using the strong threshold functions, either in
  build mode or in prune mode. Then
$$
\sum_{u\in U}d_s(u)\,\le\,(\mu(U)+4^{1-s}\delta(G,\tau))\, \tau^{r-s}\,nd
$$
holds for all subsets $U\subset[n]$ and for $1\le s\le r$.
\end{lem}
\begin{proof}
  Let $u\in U$. Just as in the proof of Lemma~\ref{lem:sigmadeg}, we have
  $d_s(u)\le \theta_s(u)+d_{s+1}(\{v,u\})$ for some~$v$. By the definition of
  $\theta_s(u)$, and by Lemma~\ref{lem:sigmadeg} applied to
  $\sigma=\{v,u\}$, we have
$$
d_s(u) \le \tau^{r-s}d(u) + 
2^{r\choose2}\tau^{r-s-1}\sum_{\ell=0}^{r-s-1} 
2^{-{s+1+\ell\choose2}+\ell}\tau^{-\ell}d^{(2+\ell)}(\{v,u\})\,.
$$
By definition we have $d^{(2+\ell)}(\{v,u\})\le d^{(2+\ell)}(u)$, and so
\begin{align*}
\sum_{u\in U}d_s(u)\,&\le\, \sum_{u\in
  U}\tau^{r-s}d(u)+ 
2^{r\choose2}\tau^{r-s-1}\sum_{\ell=0}^{r-s-1} 
2^{-{s+1+\ell\choose2}+\ell}\tau^{-\ell}\sum_{u\in U}  d^{(2+\ell)}(u)\\
&=\,\tau^{r-s}\mu(U)nd+2^{{r\choose2}-1}\tau^{r-s}\sum_{j=2}^{r-s+1} 
2^{-{s+j-2\choose2}-s+1}\tau^{1-j}\sum_{u\in U}  d^{(j)}(u)\,.
\end{align*}
Now $\tau^{1-j}\sum_{u\in U}  d^{(j)}(u)\le \tau^{1-j}\sum_{u\in [n]}d^{(j)}(u)
=\delta_jnd$. Therefore
\begin{align*}
\sum_{u\in U}d_s(u)\,&\le\,
\tau^{r-s}\mu(U)nd+\tau^{r-s}nd\,2^{{r\choose2}-1}\sum_{j=2}^{r-s+1}
2^{-{s+j-2\choose2}-s+1}  \delta_j\\
&=\,
\tau^{r-s}\mu(U)nd+\tau^{r-s}nd\,2^{{r\choose2}-1}\sum_{j=2}^{r-s+1}
2^{-{s-1\choose2}-(s-1)j-{j-1\choose2}}  \delta_j\\
&\le\,
\tau^{r-s}\mu(U)nd+4^{1-s}\tau^{r-s}nd\,\delta(G,\tau)
\end{align*}
because ${s-1\choose2}+(s-1)j\ge 2(s-1)$, and this establishes the lemma.
\end{proof}

We can now move quickly to complete the proof of
Theorem~\ref{thm:cover}. Consider a run of the algorithm in prune mode, for
some set~$I$. For $1\le s\le r-1$, let $T_s$ be the set of vertices~$v\in I$
that satisfy $v\notin B$ and $|F_{v,s}|\ge \zeta \tau^{r-s-1}d(v)$, and let
$T_0$ comprise those $v\in I$ that satisfy $v\notin B$ and $v\in \Gamma_1$. Then
$T=T_{r-1}\cup \cdots\cup T_1\cup T_0$. The sets $T_{r-1},\ldots, T_1,
T_0$, which need not be disjoint, are almost identical to the ones so named
in~\cite{ST2}; the properties of them that we need are identical, and they
hold for identical reasons. So \cite[Lemma~5.3]{ST2} shows how the
inequality $\mu(T_s) < 2\tau/\zeta$ for $s\ge1$ is easily obtained; it is
because each $v\in T_s$ contributes at least $\zeta \tau^{r-s-1}d(v)$ to
$|P_s|$, whereas $|P_s|=(1/s)\sum_{u\in[n]}d_s(u)$, which is bounded by
Lemma~\ref{lem:sigmadeg}. (When applying the lemma, note that the
conditions of Theorem~\ref{thm:cover} imply $\delta(G,\tau)<\zeta$, and we
may assume that $\zeta<1$, indeed that $\zeta<1/4r!$, else the theorem is
trivial.) Almost as direct is \cite[Lemma~5.4]{ST2}, which shows
$\mu(T_0)<2\tau/ \zeta$ under either of the sparsity constraints on $G[I]$
stated at the end of Theorem~\ref{thm:cover} (recall that if $I$ is
independent then $T_0=\emptyset$). The properties used are that if $v\in
T_0$ then $d_1(v)\ge \tau^{r-1}d(v)$, because $v\in\Gamma_1$, and for each
$\{v\}\in P_1$ there is some $(r-1)$-set $t\subset T$ with $t\cup\{v\}\in
E(G[I])$. As already noted, the same properties hold here. Thus we have
$\mu(T_s)\le 2\tau/\zeta$ for all~$s$, and so $\mu(T)<2r\tau/\zeta$. This
in turn implies $|T|<2r\tau n/\zeta^2$, because $d(v)\ge\zeta d$ for every
$v\in T$ as $v\notin B$. Consequently we obtain properties (b) and~(c) of
the theorem.

To obtain property~(d), we refer to \cite[Lemma~5.5]{ST2}, which bounds the
degree measure of the containers produced by the algorithm there. Exactly
the same bound holds here, as we now explain. The argument there is a
little intricate because of the desire to obtain a good bound on $\mu(C)$,
but the whole of it is valid here. We need only point out some slight
differences of notation. The sets $T_{r-1},\ldots, T_1, T_0$ have already
been mentioned. The container built from~$T$, which here we have denoted
$C(T)$, is there denoted $C(G,T,\tau,\zeta)$. In~\cite{ST2}, the set $C_s$
is defined, for $1\le s\le r-1$, to comprise $B$ together with those
vertices $v\notin B$ for which $|F_{v,s}|<\zeta\tau^{r-s-1}d(v)$ when the
algorithm is run in build mode. We can define $C_s$ in the same way
here. Likewise $C_0$ is there defined to be $[n]-(\Gamma_1\setminus B)$,
and we do so here. Bearing in mind that $T=T_{r-1}\cup \cdots\cup T_1\cup
T_0$, we see that the container produced by the present algorithm is
$C(T)=(C_{r-1}\cap\cdots\cap C_1\cap C_0)\cup T_{r-1}\cup \cdots\cup
T_1\cup T_0$. The container $C(G,T,\tau,\zeta)$ in~\cite{ST2} is defined in
precisely this way. The proof of \cite[Lemma~5.5]{ST2}, which gives a bound
on $\mu(C(G,T,\tau,\zeta))$, uses properties of $C_{r-1},\ldots,C_1,C_0$
and of $\Gamma_{r-1},\ldots,\Gamma_1$, together with Lemma~\ref{lem:maxdeg}
(it makes no use of $T_{r-1},\ldots, T_1, T_0$). The properties of
$C_{r-1},\ldots,C_1,C_0$ are precisely those just stated, so they hold too
in the present context, as does Lemma~\ref{lem:maxdeg}. Finally, regarding
$\Gamma_{r-1},\ldots,\Gamma_1$, the properties used are that $d_s(v)\ge
\tau^{r-s}d(v)$ if $v\in \Gamma_s$, which holds here by definition of
$\theta_s(\{v\})$, and that $d_{s-1}(\sigma)\ge 2^{s-1}\tau d_s(\sigma)$ if
$|\sigma|\ge2$ and $\sigma\in \Gamma_{s-1}$, which holds here by
Lemma~\ref{lem:sigmadeg}. We conclude that the proof of
\cite[Lemma~5.5]{ST2} carries over {\em verbatim} to give a bound on
$\mu(C(T))$.

The main theorem of~\cite{ST2} is really just a summary of the bounds
$\mu(C(G,T,\tau,\zeta))$ and on $\mu(T_{r-1}),\ldots,\mu(T_0)$, as pointed
out at the end of~\cite[\S5]{ST2}, and Theorem~\ref{thm:cover} follows in
exactly the same way.

\subsection{Weak thresholds and the proof of 
Theorem~\ref{thm:coverweak}}\label{subsec:weak}

We begin with analogues of Lemmas~\ref{lem:sigmadeg} and~\ref{lem:maxdeg}.

\begin{lem}\label{lem:sigmadegw} Let the algorithm be run using the weak
  threshold functions. Then, for $1\le s\le r$, we have
\begin{align*}
d_s(u)&\le \tau^{r-s}(d(u)+r\delta d) &\mbox{for all $u\in[n]$, and}\\
d_s(\sigma)&\le r\delta d \tau^{r-s+|\sigma|-1}&\mbox{for all
  $\sigma\subset[n]$, $2\le|\sigma|\le r$}.
\end{align*}
\end{lem}
\begin{proof}
We prove the bounds by induction on $r-s$; in fact we show
$d_s(\sigma)\le (r-s+1)\delta d \tau^{r-s+|\sigma|-1}$ for
$|\sigma|\ge2$. For $s=r$ the bounds hold by the definition
of~$\delta$ in Definition~\ref{defn:thetasw}.
Exactly as in the proof of Lemma~\ref{lem:sigmadeg}, we have
$d_s(\sigma)\le\theta_s(\sigma)+d_{s+1}(\{v\}\cup\sigma)$ for some
$v\notin\sigma$. So for $|\sigma|\ge2$ we have, by applying the induction
hypothesis to $\{v\}\cup\sigma$, 
$$
d_s(\sigma)\le \delta d\tau^{r-s+|\sigma|-1} + (r-s)\delta d
\tau^{r-s-1+|\sigma|} = (r-s+1)\delta d\tau^{r-s+|\sigma|-1}
$$
as claimed. For $\sigma=\{u\}$ we apply the induction hypothesis to
$\sigma=\{v,u\}$ to obtain
$$
d_s(u)\le \tau^{r-s}d(u)+(r-s)\delta d \tau^{r-s-1+2-1}\le
\tau^{r-s}(d(u)+r\delta d)
$$
again as claimed. This completes the proof.
\end{proof}

\begin{lem}\label{lem:maxdegw}
  Let $G$ be an $r$-graph on vertex set $[n]$ with average degree~$d$. Let
  $P_r=E(G)$ and let $P_{r-1},\ldots,P_1$ be the multisets constructed
  during the algorithm using the weak threshold functions, either in
  build mode or in prune mode. Then
$$
\sum_{u\in U}d_s(u)\,\le\,(\mu(U)+r\delta)\, \tau^{r-s}\,nd
$$
holds for all subsets $U\subset[n]$ and for $1\le s\le r$.
\end{lem}
\begin{proof} The inequalities
$$
\sum_{u\in U}d_s(u)\,\le\,\sum_{u\in U}\tau^{r-s}(d(u)+r\delta d)
\le(\mu(U)+r\delta)\, \tau^{r-s}nd
$$
follow immediately from Lemma~\ref{lem:sigmadegw} and the definition of~$\mu$.
\end{proof}

We now turn to bounds on $T$. The argument is essentially identical to the
one in~\S\ref{subsec:strongth}, which in turn is that in~\cite{ST2}, but we
give details for completeness.

\begin{lem}\label{lem:muTs}
Let $T$ be produced by the algorithm in prune mode, using weak threshold
functions. Then $\mu(T\setminus\Gamma_1)\le (r-1)(\tau/ \zeta)(1+r\delta)$.
\end{lem}
\begin{proof}
  For $1\le s\le r-1$, let $T_s=\{v\in
  T:|F_{v,s}|\ge\zeta\tau^{r-s-1}d(v)\}$. From the operation of the
  algorithm we see that $(T\setminus\Gamma_1)\subset T_1\cup\cdots\cup
  T_{r-1}$ (the sets here need not be disjoint). For each~$s$, the sets
  $F_{v,s}$ for $v\in T_s$ are added to $P_s$ and, because $P_s$ is a
  multiset, we obtain
$$
\zeta \tau^{r-s-1}nd\mu(T_s)=\zeta \tau^{r-s-1}\sum_{v\in T_s}d(v)
\le |P_s|={1\over s}\sum_{u\in [n]} d_s(u)\le {1\over
  s}\tau^{r-s}nd(1+r\delta)
$$
by Lemma~\ref{lem:maxdegw} with $U=[n]$. Thus $\mu(T_s)\le
(\tau/\zeta)(1+r\delta)$, and $\mu(T\setminus\Gamma_1)\le
\mu(T_1)+\cdots+\mu(T_{r-1}) \le (r-1)(\tau/\zeta)(1+r\delta)$.
\end{proof}

\begin{lem}\label{lem:muT0}
  Let $T$ be produced by the algorithm in prune mode, with input $I$ and
  using weak threshold functions. If $G[I]$ is $\lfloor (\zeta/r)
  \tau^{r-1}d\rfloor$-degenerate, or if $e(G[I])\le (r/\zeta)\tau^r e(G)$,
  then $\mu(T\cap\Gamma_1)\le (\tau/\zeta)(1+r\delta)$.
\end{lem}
\begin{proof}
  Write $T_0=T\cap\Gamma_1$. For each $v\in T_0$,
  $d_1(v)\ge\theta_1(v)=\tau^{r-1}d(v)$ holds because $v\in\Gamma_1$. We
  noted earlier, in~\S\ref{subsec:actual}, that each set $f\in P_s$ comes
  from some set $t\subset T$, $|t|=r-s$, such that $t\cup f \in E(G)$, and
  $f$ comprises the last $s$ vertices of the edge $t\cup f$. In particular,
  each set $\{v\}$ in the multiset $P_1$ comes from an edge $t\cup\{v\}$
  where $t\subset T$ and $v$ is the last vertex of $t\cup\{v\}$. So if
  $v\in T_0$ then the edge $t\cup\{v\}$ lies inside~$T$. Moreover there are
  $d_1(v)$ such edges with last vertex~$v$. Hence
  $$
  \tau^{r-1}nd\mu(T_0)=\tau^{r-1}\sum_{v\in T_0} d(v) \le \sum_{v\in T_0}
  d_1(v)\le e(G[T])\,.
  $$
 
  Consider first the case that $G[I]$ is $b$-degenerate, where $b\le
  (\zeta/r) \tau^{r-1}d$. Then $e(G[T])\le b|T|$, and thus $\tau^{r-1}nd\mu(T_0)
  \le e(G[T])\le b|T|\le (\zeta/r) \tau^{r-1}d|T|$, meaning
  $rnd\mu(T_0)\le \zeta d|T|$. Now if $v\in T$ then $v$
  passes the membership rule and so $v\notin B$; consequently $d(v)\ge \eta
  d$, and hence $|T|\zeta d\le \sum_{v\in T}d(v)=nd\mu(T)$. We thus have
  $rnd\mu(T_0)\le \zeta d|T| \le nd\mu(T)$, that is, $r\mu(T_0)\le
  \mu(T)$. But $\mu(T)=\mu(T\cap
  \Gamma_1)+\mu(T\setminus\Gamma_1)=\mu(T_0)+\mu(T\setminus\Gamma_1)$. 
  and so $(r-1)\mu(T_0)\le \mu(T\setminus\Gamma_1)$. The bound
  $\mu(T_0)\le (\tau/\zeta)(1+r\delta)$ now follows from
  Lemma~\ref{lem:muTs}.

  Now consider the case that $e(G[I])\le (r/\zeta)\tau^r e(G) =
  \tau^rnd/\zeta$. Then we have directly that $\tau^{r-1}nd\mu(T_0)\le
  e(G[T]) \le e(G[I])\le \tau^r nd/\zeta$, so $\mu(T_0)\le \tau/\zeta \le
  (\tau/\zeta)(1+r\delta)$.
\end{proof}

We come now to the bound on the measure $\mu(C)$ of the containers; this
bound comes from the following lemma.

\begin{lem}\label{lem:es}
Let $C$ be the set produced by the algorithm in build mode, using weak
thresholds. Let $D=([n]-C)\cup T\cup B$. Define $e_s$ by the equation
$|P_s|=e_s\tau^{r-s}nd$ for $1\le s\le r$. Then
\begin{align*}
e_{s+1} &\le\, r2^se_s+\mu(D)+\zeta+2r\delta &\mbox{for $r-1\ge s\ge 2$}\\
e_{s+1} &\le\, 2\mu(D)+\zeta+3r\delta &\mbox{for $s=1$.}
\end{align*}
\end{lem}

\begin{rem}
  Bounding $\mu(C)$ above is equivalent to bounding $\mu(D)$ from
  below. The lemma captures the spirit behind the algorithm, as discussed
  in~\cite{ST2}, that either all $P_s$ are large, in which case $\Gamma_1$
  is substantial and so $D$ is also (this is what lies behind the second
  inequality), or, for some $s$, $P_{s+1}$ is large but $P_s$ is small,
  which makes $D$ substantial by the first inequality.
\end{rem}

\begin{proof}
The way the algorithm builds $C$ means that $T\cup B\subset C$.
Let $C'=C-(T\cup B)$, so $D=[n]-C'$. Again, by the operation of the
algorithm, if $v\in \Gamma_1\setminus (T\cup B)$ then $v\notin C$, which
means that $\Gamma_1\subset D$.

For $v\in[n]$ let $f_{s+1}(v)$ be the number of sets in $P_{s+1}$ for
which $v$ is the first vertex in the vertex ordering. Then
\begin{equation}
|P_{s+1}|=\sum_{v\in [n]}f_{s+1}(v)=\sum_{v\in C'}f_{s+1}(v)+
\sum_{v\in D}f_{s+1}(v)
\qquad\mbox{for $1\le s<r$.}\label{eqn:Ps}
\end{equation}
By definition of $|F_{v,s}|$, of the $f_{s+1}(v)$ sets in $P_{s+1}$ beginning
with~$v$, $f_{s+1}(v)-|F_{v,s}|$ of them contain some $\sigma\in\Gamma_s$.  If
$v\in C'$ then $v\notin \Gamma_1$, $v\notin B$ and $v\notin T$ and so,
since $v\in C$, we have $|F_{v,s}|<\zeta\tau^{r-s-1}d(v)$. Therefore, writing
$P\Gamma$ for the multiset of edges in $P_{s+1}$ that contain some $\sigma\in
\Gamma_s$, we have
\begin{equation}
\sum_{v\in C'}f_{s+1}(v)-\zeta\tau^{r-s-1}d(v)< |P\Gamma|\le
\sum_{\sigma\in \Gamma_s} d_{s+1}(\sigma)\,.\label{eqn:Pgamma}
\end{equation}
By definition, if $\sigma\in\Gamma_s$ and $|\sigma|\ge2$, then $d_s(\sigma)\ge
\theta_s(\sigma)= \delta d\tau^{r-s+|\sigma|-1}$. Using
Lemma~\ref{lem:sigmadegw}, we then see that $d_{s+1}(\sigma)\le r\delta
d\tau^{r-s+|\sigma|-2}\le (r/\tau)d_s(\sigma)$. Similarly, if
$\sigma=\{u\}\in \Gamma_s$ then $d_s(\sigma)\ge \tau^{r-s}d(u)$ and
$d_{s+1}(\sigma)\le \tau^{r-s-1}(d(u)+r\delta d)\le (1/\tau) d_s(\sigma) +
r\delta d\tau^{r-s-1}$. Therefore, for $s\ge2$, we obtain
\begin{equation}
\sum_{\sigma\in \Gamma_s} d_{s+1}(\sigma)
\le {r\over\tau} \sum_{\sigma\in \Gamma_s} d_s(\sigma)
+\sum_{\{u\}\in \Gamma_s}r\delta d\tau^{r-s-1}
\le {r\over\tau}2^s|P_s| + r\delta nd\tau^{r-s-1}\,.\label{eqn:ss}
\end{equation}
For $s=1$ the possibility $|\sigma|\ge2$ does not arise, and we obtain
\begin{align}
\sum_{\sigma\in \Gamma_s} d_{s+1}(\sigma)
&\le \sum_{\{u\}\in \Gamma_1}\left({1\over \tau}d_1(u)+r\delta
  d\tau^{r-2}\right)\notag\\
&\le \tau^{r-2}nd(\mu(\Gamma_1)+r\delta) + r\delta n  d\tau^{r-2}
&\mbox{by Lemma~\ref{lem:maxdegw}}\notag\\
&\le\tau^{r-2}nd\mu(D)+2r\delta n  d\tau^{r-2} &\mbox{since
  $\Gamma_1\subset D$.}\label{eqn:s2}
\end{align}
Finally, making use of~(\ref{eqn:Ps}) and~(\ref{eqn:Pgamma}) together with
Lemma~\ref{lem:maxdegw}, we have
\begin{align*}
e_{s+1}\tau^{r-s-1}nd=|P_{s+1}| &=\sum_{v\in C'}f_{s+1}(v)+
\sum_{v\in D}f_{s+1}(v)\\
&\le \sum_{v\in C'}\zeta\tau^{r-s-1}d(v) + 
\sum_{\sigma\in \Gamma_s} d_{s+1}(\sigma)+ \sum_{v\in D}d_{s+1}(v) \\
&\le \zeta\tau^{r-s-1}nd + 
\sum_{\sigma\in \Gamma_s} d_{s+1}(\sigma)+ \tau^{r-s-1}nd(\mu(D)+r\delta)\,,
\end{align*}
The bounds~(\ref{eqn:ss}) and~(\ref{eqn:s2}) for 
$\sum_{\sigma\in \Gamma_s} d_{s+1}(\sigma)$ now give the result claimed.
\end{proof}

\begin{proof}[Proof of Theorem~\ref{thm:coverweak}]
  We begin by choosing the constant $c=c(r)$. Let
  $\gamma=(1/25)r^{-2r}2^{-r^2}$ and $c=\gamma^r$.  Let $G$ be as in the
  theorem and let $\tau$ be chosen so that (\ref{eqn:dag}) is satisfied.
  Let $\zeta=\sqrt{2r\gamma}$. For later use, we note $c\le \gamma \le
  \zeta/2r \le 2r\zeta \le 1$. 

  As might be expected, we prove the theorem by using the containers~$C$
  and the sets~$T$ supplied by the algorithm, using the weak threshold
  functions. However, the input parameters we supply to the algorithm are
  not $\tau$ and $\zeta$ as just defined, but instead $\tau_*=\gamma\tau$
  and~$\zeta$.

  The reason for using slightly different parameters is to obtain as clean
  a statement of Theorem~\ref{thm:coverweak} as possible, subject to not
  worrying about the best value of~$c$. For example, assertion~(b) of
  the theorem states $\mu(T)\le \tau$, whereas the corresponding assertion
  of Theorem~\ref{thm:cover} states $\mu(T)\le 2r\tau/\zeta$. What, in
  effect, we achieve by using $\tau^*$ instead of $\tau$ is that we shall
  obtain $\mu(T)\le 2r\tau^*/\zeta$, which (as we shall check) implies
  $\mu(T)\le\tau$. In a similar manner, all the other constants that might
  otherwise appear in the statement of the theorem are absorbed into the
  small constant~$c$.

  We therefore remind the reader that the values of~$\tau$ and~$\zeta$
  appearing in the lemmas above are those values input to the algorithm. This
  was highlighted at the start of~\S\ref{sec:analysis}. Hence in the
  present case, where we are using inputs $\tau^*$ and~$\zeta$, the
  conclusions of the lemmas hold with $\tau^*$ in place of~$\tau$. Again,
  as highlighted earlier, the value of $\delta$ in the lemmas is that
  supplied by Definition~\ref{defn:thetasw} with $\tau^*$ in place
  of~$\tau$. Explicitly, $\delta$ is (by definition) minimal such that
  $d(\sigma)\le\delta d\tau^{*(|\sigma|-1)}$ for all~$\sigma$. Now $\tau$
  was chosen to satisfy~(\ref{eqn:dag}), so we know that $d(\sigma)\le cd
  \tau^{(|\sigma|-1)}$. Since $c=\gamma^r$ this implies we know, for
  all~$\sigma$, that $d(\sigma)\le \gamma^r d \tau^{(|\sigma|-1)}\le \gamma
  d\tau^{*(|\sigma|-1)}$, because $\gamma\le 1$ and $|\sigma|\le
  r$. Consequently, by the minimality of~$\delta$, we have $\delta\le
  \gamma$.

  What remains is to verify the claims of the theorem. Condition~(a)
  follows from the general properties of the algorithm, as discussed
  in~\S\ref{subsec:goldenrule}, and the online property follows too, as
  explained in Remark~\ref{rem:online}.

  We know that either $G[I]$ is $\lfloor c \tau^{r-1} d\rfloor$-degenerate
  or $e(G[I])\le c\tau^re(G)$. Now $c\tau^{r-1}=\gamma\tau_*^{r-1}\le
  (\zeta/r)\tau_*^{r-1}$, and $c\tau^r=\tau_*^r\le (r/\zeta) \tau_*^r$. Hence
  the conditions of Lemma~\ref{lem:muT0} are satisfied. So, by
  Lemmas~\ref{lem:muTs} and~\ref{lem:muT0},
  $\mu(T)=\mu(T\setminus\Gamma_1)+\mu(T\cap\Gamma_1) \le
  (r\tau_*/\zeta)(1+r\delta)\le 2r\tau_*/\zeta=2r\gamma\tau/\zeta =
  \zeta\tau$, easily establishing condition~(b). Moreover $T\cap
  B=\emptyset$, so $|T|\zeta d\le\sum_{v\in T}d(v) = nd\mu(T)\le
  nd\zeta\tau$, giving condition~(c).

  To show that condition~(d) holds, note that
  $2r\delta\le2r\gamma\le\zeta$, and so by Lemma~\ref{lem:es} we
  comfortably have
  \begin{align*}
    e_{s+1} &\le\, r2^se_s+\mu(D)+2\zeta &\mbox{for $r-1\ge s\ge 2$}\\
    e_{s+1} &\le\, 2\mu(D)+4\zeta &\mbox{for $s=1$.}
  \end{align*}
  Dividing the bound for $e_{s+1}$ by $r^{s+1}2^{s+1\choose2}$ and adding over
  $s=1,\ldots,r-1$, we obtain
  $$
  {e_r\over r^r2^{r\choose2}}\le (\mu(D)+2\zeta)\left\{{1\over r^2}+{1\over
      r^3}{1\over 2^3}+{1\over r^4}{1\over 2^6}+\cdots\right\} \le
  (\mu(D)+2\zeta){2\over r^2}\,.
  $$
  Recall that $e_rnd=|P_r|=e(G)=nd/r$ so $e_r=1/r$. Hence $\mu(D)+2\zeta\ge
  r^{-r}2^{-{r\choose2}}=5 \gamma^{1/2} 2^{r/2}\ge 5\zeta$. So $\mu(D)\ge
  3\zeta$.  By definition, $D=[n]-(C-(T\cup B))$. Thus $\mu(C)\le
  1-\mu(D)+\mu(T)+\mu(B)$. We showed previously that $\mu(T)\le\zeta\tau$,
  so $\mu(T)\le\zeta$ because $\tau\le1$. Moreover $\mu(B)\le \zeta$ by
  definition of $B$. Therefore $\mu(C)\le 1-3\zeta+\zeta+\zeta = 1-\zeta\le
  1-c$, completing the proof.
\end{proof}

We finish with a proof of Corollary~\ref{cor:sparse_container}.

\begin{proof}[Proof of Corollary~\ref{cor:sparse_container}]
  Write $c_*$ for the constant $c(r)$ from Theorem~\ref{thm:coverweak}. We
  prove the corollary with $c=\epsilon\ell^{-r}c_*$, where
  $\ell = \lceil(\log\epsilon)/\log(1-c_*)\rceil$. Let $G$, $I$ and $\tau$
  be as stated in the corollary. We shall apply Theorem~\ref{thm:coverweak}
  several times. Each time we apply the theorem, we do so with with
  $\tau_*=\tau/\ell$ in place of $\tau$, with the same $I$, but with
  different graphs $G$, as follows (we leave it till later to check that the
  necessary conditions always hold). Given $I$, apply the theorem to find
  $T_1\subset I$ and $I\subset C_1=C(T_1)$, where $|T_1|\le \tau_* n$ and
  $\mu(C_1)\le 1-c_*$. It is easily shown that $e(G[C_1])\le \mu(C_1)e(G)\le
  (1-c_*)e(G)$ (this is \cite[inequality~(1)]{ST2}). Now $I$ is sparse in
  the graph $G[C_1]$ so apply the theorem again, to the $r$-graph $G[C_1]$,
  to find $T_2\subset I$ and a container $I\subset C_2$.  We have $|T_2|\le
  \tau_*|C_1|$, and $e(G[C_2])\le (1-c_*)e(G[C_1])\le (1-c_*)^2e(G)$. By
  Remark~\ref{rem:enlarge}, we note that, in the first application, 
  the algorithm in build mode
  would have constructed $C_1$ from input $T_1\cup T_2$, and would likewise
  have constructed $C_2$ from input $T_1\cup T_2$ in the second
  application. Thus $C_2$ is a function of $T_1\cup T_2$. We repeat this
  process $k$ times until we obtain the desired container $C=C_k$ with
  $e(G[C])\le \epsilon e(G)$. Since $e(G[C])\le (1-c_*)^ke(G)$ this occurs
  with $k\le\ell$. Put $T=T_1\cup\cdots\cup T_k$. Then $C$ is a function of
  $T\subset I$.

  We must check that the requirements of Theorem~\ref{thm:coverweak}
  are fulfilled at each application. Observe that, if $d_j$ is the average
  degree of $G[C_j]$ for $j<k$, then $|C_j|d_j=re(G[C_j])> r\epsilon
  e(G)=\epsilon nd$, and since $|C_j|\le n$ we have $d_j\ge\epsilon d$. The
  conditions of Corollary~\ref{cor:sparse_container} mean that
  $d(\sigma)\le c d\tau^{|\sigma|-1} =\epsilon \ell^{-r}c_*d\tau^{|\sigma|-1} <
  c_*d_j\tau_*^{|\sigma|-1}$; since the degree of $\sigma$ in $G[C_j]$ is at
  most $d(\sigma)$, this means that (\ref{eqn:dag}) is satisfied every time
  Theorem~\ref{thm:coverweak} is applied.

  Observe that, because $d_j\ge \epsilon d$, $e(G[C_j])>\epsilon e(G)$ and
  $c= \epsilon \ell^{-r}c_*$, then $G[I]$ is $\lfloor c_* \tau_*^{r-1}
  d_j\rfloor$-degenerate if it is $\lfloor c \tau^{r-1}
  d\rfloor$-degenerate, and $e(G[I])\le c_*\tau_*^re(G[C_j])$ if $e(G[I])\le
  c\tau^re(G)$. Therefore the theorem is being applied correctly each time.

  Finally condition~(c) of the theorem implies $|T_j|\le \tau_*|C_j|\le
  \tau_*n = \tau n/\ell$, and so $|T|\le k\tau n /\ell\le \tau n$,
  giving condition~(b) of the corollary and completing the proof.
\end{proof}

\section{Linear equations}\label{sec:sumfree}

Recall from~\S\ref{subsec:lineq} the definitions of a linear system
$(F,A,b)$, of $\mbox{ex}(F,A,b)$, of \emph{full rank} and of
\emph{abundant}.

Often one wishes to discount solutions to an equation $Ax=b$ where the
vector $x$ contains repeated values. For example, in forbidding a 3-term
arithmetic progression, we take $A = (1,1,-2)$ and $b=(0)$ and discount
solutions of the form $x + x - 2x = 0$. To accommodate this setup, we let
$Z \subset F^r$ be a set of discounted solutions. We then call $(F,A,b,Z)$
a \emph{$k\times r$ linear system}. A solution to this system is a vector
$x\in F^r - Z$ such that $Ax=b$, and a set $I\subset F$ is {\em
  solution-free} if there is no $x\in I^r - Z$ such that $Ax=b$.

In order to state the main theorem about linear systems, we need to define
the following parameter, following R\"odl and Ruci\'nski~\cite{RR}.

\begin{defn}\label{def:mA}
{\em If $F$ is a finite field  or $[N]$}, and $A$ is an abundant
 $k\times r$ matrix over~$F$, then we define
\[
 m_F(A) = \max_{J \subset [r],\, |J| \ge 2}
          \frac{|J|-1}{|J|-1+\mbox{rank}(A_J)-k},
\]
where the matrix $A_J$ is the $k\times(r-|J|)$ submatrix
of $A$ obtained by deleting columns indexed by $J$.
{\em If $F$ is an abelian group}, and $A$ is an abundant
 $k\times r$ integer matrix, then let $t$ be the maximum value of $j$
for which $A_J$ has full rank whenever $|J|=j$, and define 
\[
 m_F(A) = \frac{k+t-1}{t-1}.
\]
\end{defn}

It can readily be checked that if $A$ is abundant then the denominators
appearing in the definition of $m_F(A)$ are strictly positive.
The separate definition of $m_F(A)$ when $F$ is an abelian group is necessary
since the rank of an integer matrix over an abelian group is not well-defined;
in general, when the pair $(F,A)$ could either be considered a finite field
or an abelian group with $A$ integer valued, the value of the second definition
is at least as big as the value of the first definition. This is since
$\mbox{rank}(A_J)=k$ when $|J|\le t$, and is otherwise at least
$\max\{0,k+t-|J|\}$.

From our point of view, the parameter $m_F(A)$ plays a role for
solution-free sets very similar to the role that the parameter $m(H)$ plays
for $H$-free graphs. Our main theorem here about linear systems,
Theorem~\ref{thm:eqn_cover}, gives containers for solution-free subsets,
and the number of containers depends on~$m_F(A)$; in like manner, our main
theorem about $H$-free graphs, \cite[Theorem~2.3]{ST2}, gives containers
for $H$-free graphs and the number of them depends on~$m(H)$. All our
further results about linear systems, for example
Theorem~\ref{thm:eqn_sparse} about sparse systems, are applications of
Theorem~\ref{thm:eqn_cover}, in the same way that all the results
in~\cite{ST2} about $H$-free graphs, such as \cite[Theorem~2.12]{ST2} for
sparse graphs, are applications of \cite[Theorem~2.3]{ST2}.

To understand why the parameter $m(H)$ takes the form it does, it is
easiest to look at the application to sparse graphs, where a simple
argument shows that \cite[Theorem~2.12]{ST2} is best possible, and so, by
implication, the number of containers in \cite[Theorem~2.3]{ST2} must
depend on~$m(H)$. To illuminate the form of the parameter $m_F(A)$, a
similar argument can be put forward for sparse linear systems, showing that
Theorem~\ref{thm:eqn_sparse} is best possible. However, it does not apply
in every case, because sometimes extra conditions are needed (discussed by
R\"odl and Ruci\'nski~\cite{RR}). Fortunately, these extra conditions play
no part in Theorem~\ref{thm:eqn_cover}.

The argument is as follows. Assume, for simplicity, that the system is
$Ax=0$ and that $F$ is a finite field. We show that, when a random subset
$X\subset F$ is selected by choosing elements each with probability~$p$,
and $p$ is substantially smaller than $|F|^{-1/m_F(A)}$, then there is
(almost surely) a solution-free subset $X^*\subset X$ which is nearly as
large as~$X$. In this sense Theorem~\ref{thm:eqn_sparse} is best
possible. Let $A_J$ be a maximizing submatrix in the definition of $m_F(A)$
and let $B_J$ be the $k\times|J|$ submatrix deleted from $A$ to form~$A_J$.
Write $\langle A_J\rangle$, $\langle B_J\rangle$ for the spaces spanned by
the columns of $A_J$, $B_j$ respectively, and let their dimensions be
$\alpha=\mbox{rank}(A_J)$ and $\beta=\mbox{rank}(B_J)$.  Let $W=\langle
A_J\rangle\cap \langle B_J\rangle$. Since the sum $\langle
A_J\rangle+\langle B_J\rangle$ is the space spanned by the columns of the
full rank matrix~$A$, it has dimension~$k$, and so
$\mbox{dim}(W)=\alpha+\beta-k$. For each $x\in F^r$, let $x'$ be its
projection onto the coordinates indexed by $J$. If $Ax=0$ then $B_Jx'\in
\langle A_J\rangle$, and so $B_Jx'\in W$. For each vector $w\in W$ there
are at most $|F|^{|J|-\beta}$ solutions $x'$ to $B_Jx'=w$ (see
Fact~\ref{fact:eqn_deg_codeg}), so if $V=\{x':Ax=0\}$ then
$|V|\le|F|^{|J|-\beta}|W|=|F|^{|J|-\beta+{\rm dim}(W)} =
|F|^{|J|+\alpha-k}$.  Let $X$ be chosen randomly as just described, with
$p$ much smaller than $|F|^{-1/m_F(A)}$. Then $|X|$ will likely be near
$p|F|$, and the number of vectors $x'\in V$ lying within $X^k$ is unlikely
to be much larger than $p^{|J|}|F|^{|J|+\alpha-k}$. Since $m_F(A)=
(|J|-1)/(|J|-1+\alpha-k)$, this number is much smaller than~$|X|\approx
p|F|$, and by removing from $X$ an element of each such~$x'$, we obtain a
subset $X^*\subset X$, with $|X^*|$ close to~$|X|$, such that $X^*$
contains no solution $x'$ with $B_Jx'\in W$ and so $X^*$ is solution-free
for the system $Ax=0$.

After all these preliminaries, we now state the main theorem on linear
systems.

\begin{thm}\label{thm:eqn_cover}
  Let $(F,A,b,Z)$ be a $k\times r$ linear system with $A$ abundant and
  $|Z|\le|F|^{r-k}/2$. Given $\epsilon>0$ there is a constant $c$,
  depending on $A,\epsilon$ in the case $F=[N]$, and depending only on
  $k,r,\epsilon$ otherwise, such that if $|F|\ge c$ then there exists $\C
  \subset \mathcal{P}F$ satisfying
\begin{itemize}
 \item[(a)]
   for every solution-free subset $I\subset F$ there exists 
   $T\subset I$ such that $I\subset  C=C(T)\in\C$, and  $|T| \le c
   |F|^{1-1/m_F(A)}$,
 \item[(b)] for every $C \in \C$, the number of  solutions to $Ax = b$
   with $x \in C^r-Z$ is at most $\epsilon |F|^{r-k}$,
 \item[(c)] $\log |\C| \leq c |F|^{1-1/m_F(A)} \log |F|$.
\end{itemize}
\end{thm}

%\pagebreak

The theorem is just a straightforward consequence of
Corollary~\ref{cor:sparse_container}. It is necessary only to construct a
suitable hypergraph that encodes solutions to the linear system, and then
to check its parameters so that the corollary can be applied. The
hypergraph in question is the following.

\begin{defn}\label{def:GFAb}
  Let $(F,A,b,Z)$ be a $k\times r$ linear system.  The $r$-partite
  $r$-graph $G=G(F,A,b,Z)$ has vertex set $V(G) = X_1 \cup \cdots \cup
  X_r$, where the $X_i$s are disjoint copies of $F$, and edge set $E(G) =
  \{x=(x_1, \ldots, x_r) \in X_1 \times \cdots \times X_r - Z : Ax = b\}$.
\end{defn}

In order to apply Corollary~\ref{cor:sparse_container} to $G(F,A,b,Z)$ we
need to estimate the quantities $d(\sigma)$, which we now proceed to do.

\begin{fact}\label{fact:eqn_deg_codeg}
  Let $F$ be a finite field or abelian group, let $A$ be a $k \times \ell$
  matrix and let $b \in F^k$. If $A$ has full rank then there are
  $|F|^{\ell-k}$ solutions to $Ax=b$. More generally if $F$ is a finite
  field, there are at most $|F|^{\ell-{\rm rank}(A)}$ solutions to $Ax=b$.
\end{fact}

\begin{proof}
  If $A$ has full rank, then for every $b_1, b_2 \in F^k$ there exists $x
  \in F^\ell$ with $A x = b_2 - b_1$.  Thus if $x_1$ is a solution to $A
  x_1 = b_1$ then $A (x_1 + x) = b_2$, so by symmetry every $b \in F^k$ has
  $|F|^\ell / |F|^k$ solutions to $A x = b$.  The case when $F$ is a finite
  field is standard.
\end{proof}

\begin{lem}\label{lem:eqn_delta}
  Let $(F,A,b,Z)$ be a $k\times r$ linear system where $F$ is a finite
  field or an abelian group, $A$ is an abundant matrix and
  $|Z|\le|F|^{r-k}/2$.  Let $G=G(F,A,b,Z)$, $\gamma\le1$ and $\tau =
  |F|^{-1/m_F(A)}/\gamma$.  Then $d(\sigma) \le 2\gamma d\tau^{|\sigma|-1}$
  holds for every $\sigma\subset V(G)$ with $2\le|\sigma|\le r$, where $d$
  is the average degree of~$G$.
\end{lem}

\begin{proof}
  The number of edges in $G$ is the number of solutions to $Ax = b$ not in
  $Z$.  The matrix $A$ has full rank, so by Fact~\ref{fact:eqn_deg_codeg}
  the number of edges of $G$ is $|F|^{r-k}-|Z|\ge|F|^{r-k}/2$. Since $G$
  has $r|F|$ vertices, its average degree $d$ satisfies $d\ge
  |F|^{r-k-1}/2$.

  Let $\sigma\subset V(G)$ where $2\le|\sigma|\le r$. Put $j=|\sigma|$. If
  $\sigma$ contains two vertices in the same part $X_i$ then there are no
  edges containing $\sigma$.  Otherwise, we may suppose that $\sigma=\{y_1,
  \ldots, y_j\}$, where $y_\ell \in X_{i_\ell}$ for $\ell=1,\ldots,j$.
  Let $J=\{i_1,\ldots,i_j\}$. Then
  $d(\sigma)$ is at most the number of solutions to $Ax=b$ with $x_{i_\ell}
  = y_\ell$ for $\ell=1,\ldots,j$, and there is some $b^*\in F^k$ for which
  this is the number of solutions $x^* \in F^{r-j}$ to $A_Jx^*=b^*$.  We
  now split the proof into two cases depending on whether $F$ is a finite
  field or an abelian group.

  {\sl When $F$ is an abelian group:} Recall Definition~\ref{def:mA}, and
  in particular that $m_F(A)= (k+t-1)/(t-1)$. Write $f=|F|^{-1/m_F(A)}$, so
  $\tau=f/\gamma$. Note that $|F|f\ge1$. If $j\le t$ then $A_J$ has full
  rank by assumption, and so Fact~\ref{fact:eqn_deg_codeg} implies the
  number of solutions is at most $|F|^{r-j-k}$. Hence for $2 \le j \le t$ we
  have
  \[
  \frac{d(\sigma)}{d\tau^{|\sigma|-1}}
  \le 2|F|^{r-j-k - (r-k-1)} \tau^{1-j}= 2|F|^{1-j}\tau^{1-j}
  = 2\gamma^{j-1} (|F|f)^{1-j}\le 2\gamma\,,
  \]
  since $\gamma\le1$ and $|F|f\ge1$.
  When $t+1 \le j \le t+k$, we can say $d(\sigma)\le d(\sigma')$ for some
  $\sigma'\subset\sigma$ with $|\sigma'|=t$, so
  \[
  \frac{d(\sigma)}{d\tau^{|\sigma|-1}} \le 
  \frac{d(\sigma')}{d\tau^{j-1}}
  \le 2|F|^{1-t}\tau^{1-j}\le 2\gamma|F|^{1-t}f^{1-j}\le
  2\gamma|F|^{1-t}f^{1-t-k}= 2\gamma\,,
  \]
  here using $\gamma\le1$ and the definition of~$f$.  When $t+k< j \le
  r$, the crude bound $d(\sigma)\le |F|^{r-j}$ (recall we
  are counting solutions $x^* \in F^{r-j}$) is enough. Using $\gamma\le1$
  and $|F|f\ge1$ we have
  \[
  \frac{d(\sigma)}{d\tau^{|\sigma|-1}} \le 
  2|F|^{r-j-(r-k-1)}\tau^{1-j} \le 2\gamma|F|^k (|F|f)^{1-j} 
  < 2\gamma|F|^k (|F|f)^{1-t-k}=2\gamma\,.
  \]
  Therefore $d(\sigma)/d\tau^{|\sigma|-1}\le
  2\gamma$ for all $j$, as claimed.

  {\sl When $F$ is a finite field:} By Fact~\ref{fact:eqn_deg_codeg} the
  number of solutions to $A_Jx^*=b^*$ is at most
  $|F|^{r-j-{\rm rank}(A_J)}$. Hence
  \[
  d(\sigma) \le \max_{J \subset [r],\, |J|=j} |F|^{r-j-{\rm rank}(A_J)}.
  \]
  Using $\tau=\gamma^{-1}|F|^{-1/m_F(A)}$ and $\gamma\le1$, this implies
  that
  \[
  \frac{d(\sigma)}{d\tau^{|\sigma|-1}} \le 2\gamma\max_{J \subset [r],\, |J|=j}
  |F|^{1-j+k-{\rm rank}(A_J)+(j-1)/m_F(A)}.
  \]
  The exponent is at most $0$ by definition of $m_F(A)$, so
  $d(\sigma)/d\tau^{|\sigma|-1}\le 2\gamma$.
\end{proof}

\begin{proof}[Proof of Theorem~\ref{thm:eqn_cover}]
  We may assume that $F$ is a finite field or abelian group.  Indeed, $[N]$
  can be embedded into the finite field $\mathbb{Z}_p$ for a sufficiently
  large prime~$p$. Taking $p$ in the range $4k!|A|^k N \le p \le 8k!|A|^k
  N$, where $|A|$ is the sum of the absolute values of the entries of $A$,
  guarantees that $A$ is still abundant in $\mathbb{Z}_p$ and that a
  solution to $Ax=b\mbox{ (mod $p$)}$ is also a solution to $Ax=b$
  (provided, say, $|b_i| \le p/2$; but we may assume this since otherwise
  there are no solutions to $Ax=b$ in $[N]$).  Then the result of this
  theorem for $(\mathbb{Z}_p,A,b,Z)$ implies the result for $([N],A,b,Z)$,
  since $p/N$ is bounded by a constant depending only on $A$.

  Let $c'=c(r,\epsilon)$ be the constant asserted by
  Corollary~\ref{cor:sparse_container}. Choose $\gamma\le1$ so that
  $2\gamma\le c'$, and put $c=\max\{(1/\gamma)^k,2r/\gamma\}$. We claim
  this~$c$ works in the theorem.

  To see this, apply the corollary to the $r$-graph $G = G(F,A,b,Z)$, with
  $\tau=|F|^{-1/m_F(A)} / \gamma$. If $|F|\ge c$ then $|F|\ge
  (1/\gamma)^{m_F(A)}$ and so $\tau\le 1$. By Lemma~\ref{lem:eqn_delta},
  the requirements of Corollary~\ref{cor:sparse_container} are then
  satisfied.  So we obtain a collection of sets $\mathcal{D}$ covering the
  independent sets of $G$.  For $D \in \mathcal{D}$, let $\pi_i(D) = D\cap
  X_i \subset F$ be the part of $D$ in the $i$th copy of $F$ and let
  \[
  \C = \{ C_D : D \in \mathcal{D} \} \subset \mathcal{P}F \qquad
  \mbox{where } C_D=\pi_1(D) \cap \cdots \cap \pi_r(D).
  \]
  We claim that $\C$ satisfies the conditions of the theorem.

  Condition (a): consider a solution-free set $I \subset F$.  The subset
  $J$ of $V(G)$ formed by taking a copy of $I$ in each $X_i$ is an
  independent set in $G$.  In particular, it is contained in some $D \in
  \mathcal{D}$, hence $I \subset C_D$. Moreover, there exists a set
  $T'\subset J$, such that $D=D(T')$ and $|T'| \le \tau |G|$. Let
  $T=\pi_1(T')\cup\cdots\cup\pi_r(T')$; then $|T|\le|T'|\le (r/\gamma)
  |F|^{1-1/m_F(A)}<c|F|^{1-1/m_F(A)}$. Now let $S'$ be the
  subset of $V(G)$ formed by taking a copy of $T$ in each~$X_i$; clearly
  $T$ determines $S'$. By
  definition of $J$ we have $T\subset I$ and $T'\subset S'\subset J$. By
  Remark~\ref{rem:enlarge}, which describes the iterative process leading
  from Theorem~\ref{thm:coverweak} to Corollary~\ref{cor:sparse_container}, we
  know that $D(S')=D(T')=D$, and therefore $T$ determines $C_D$. This
  verifies condition~(a).

  Condition (b): consider $C\in\C$. Each solution to $Ax = b$ with $x \in
  C^r - Z$ corresponds to an edge of $G[D]$, of which there are at most
  $\epsilon e(G)=\epsilon |F|^{r-k}$.

  Condition (c): putting $q=(r/\gamma) |F|^{1-1/m_F(A)}$, so $|T|\le q$,
  we have $|\C|\le \sum_{t\le q}|F|^t\le (q+1)|F|^q<|F|^{2q}$. Thus $\log
  |\C| \le 2q\log|F|\le c |F|^{1-1/m_F(A)} \log |F|$, completing the proof.
\end{proof}

\subsection{Supersaturation}\label{subsec:supersat}

Condition~(b) of Theorem~\ref{thm:eqn_cover} provides containers $C$ that
contain few solutions. Our applications require a bound on $|C|$ itself. We
obtain such a bound from condition~(b) via the notion of supersaturation,
as given in Definition~\ref{defn:supersat}. The name is taken from the
supersaturation theorem of Erd\H{o}s and Simonovits~\cite{ES}, which proves
a similar property for hypergraphs and other discrete structures by a
simple averaging argument.

For some linear systems $(F,A,b)$ it is possible to prove supersaturation
by an averaging argument of this kind. For example, consider arithmetic
progressions of length~$\ell$ in~$F=[N]$; these are solutions to $Ax=0$ for
some $((\ell-2)\times\ell)$-matrix~$A$. Szem\'eredi's theorem~\cite{Sz}
shows that $\mbox{ex}(F,A,0)=o(|F|)$, from which Varnavides~\cite{Va}
derived (for $l=3$, but it works in general) that $|X|<\epsilon N$ if
$X\subset [N]$ contains fewer than $\delta(\epsilon)N^2$ arithmetic
progressions. That is, $(F,A,0)$ is $f$-supersaturated for some null~$f$
not depending on~$F$.

Such a simple averaging argument does not usually work, and we
might then turn to a removal lemma. This is stronger than the
supersaturation condition: it says that if $X\subset F$ contains at most
$\eta|F|^{r-k}$ solutions to $Ax=b$ then there is a subset $X'\subset X$,
$|X'|<\epsilon |F|$, such that $X-X'$ is solution-free. The archetype for
such lemmas is the Triangle Removal Lemma of Ruzsa and
Szemer\'edi~\cite{RS}.

Green~\cite{G} proved a removal lemma for single linear equations over
abelian groups. He conjectured a similar lemma for systems over a finite
field, which was proved by Shapira~\cite{Sh} and by Kr\'al', Serra and
Vena~\cite{KSV1}. These proofs use removal lemmas for hypergraphs such as
those of Austin and Tao~\cite{AT}, Gowers~\cite{Go}, Nagle, R\"odl and
Schacht~\cite{NRS} and Tao~\cite{T}. In fact, Szegedy~\cite{Szeg} pointed
out that, subject to certain symmetry conditions, hypergraph removal lemmas
can lead directly to algebraic removal lemmas.

Kr\'al', Serra and Vena~\cite{KSV2} also give a version for systems over
abelian groups. The statement involves the \emph{determinantal} of a
$k\times r$ integer matrix, which is the greatest common divisor of the
determinants of its $k\times k$ submatrices; note that if~$A$ has
determinantal coprime to~$|F|$ then in particular~$A$ has full rank. We do
not quote the removal lemma exactly, but rather its consequence for
supersaturation.

\begin{prop}[Kr\'al', Serra and
  Vena~\cite{KSV1,KSV2}]\label{prop:eqn_removal}
  Let $k,r\in\mathbb{N}$. Then there is a null function
  $f:\mathbb{R}^+\to\mathbb{R}^+$ such that, if $(F,A,b)$ is a $k \times
  r$ linear system where $F$ is a finite field or abelian group and~$A$ has
  full rank, and if further~$A$ has determinantal coprime to~$|F|$ in the
  case that~$F$ is an abelian group, then $(F,A,b)$ is $f$-supersaturated.
\end{prop}

We thus have a wide class of $f$-supersaturated linear systems where $f$
depends only on $k$ and $r$.

\subsection{A couple of applications}

We begin with a strengthened version of Theorem~\ref{thm:eqn_count} which
takes into account a set $Z$ of discounted solutions.

\begin{thm}\label{thm:eqn_countz}
  Let $k,r\in\mathbb{N}$ and let $f:\mathbb{R}^+\to\mathbb{R}^+$ be null.
  Let $(F,A,b,Z)$ be a $k \times r$ linear system with $A$ abundant and
  $(F,A,b)$ $f$-supersaturated. Given $\epsilon>0$, there exists
  $c=c(k,r,f,\epsilon)$ (or $c=c(A,f,\epsilon)$ in the case $F=[N]$) and
  $\eta=\eta(f,\epsilon)>0$, such that, if $|F|>c$ and $|Z|<\eta|F|^{r-k}$,
  then the number of solution-free subsets of $F$ is
  $2^{{\rm ex}(F,A,b)+\lambda|F|}$, where $0\le\lambda<\epsilon$.
\end{thm}
\begin{proof}
There is a set of size ${\rm ex}(F,A,b)$ containing no solution to the
system $(F,A,b)$ and therefore certainly no solution to the system
$(F,A,b,Z)$. Every subset of this set is a solution-free subset for the
system $(F,A,b,Z)$, so we obtain $2^{{\rm ex}(F,A,b)}$ such subsets. This
proves $\lambda\ge0$.

To obtain the upper bound $\lambda<\epsilon$, let $\eta>0$ be such that
$f(2\eta)<\epsilon/2$, which exists because $f$ is null. Let $c$ be the
constant supplied by Theorem~\ref{thm:eqn_cover} when $\eta$ is used in
place of~$\epsilon$. Then $(F,A,b,Z)$ satisfies the conditions of the
theorem (we can of course assume $\eta<1/2$) so we obtain a collection
$\mathcal{C}$ of containers for the solution-free subsets.

By increasing $c$ if necessary, condition~(c) of the theorem implies
$\log|\mathcal{C}|\le (\epsilon/2)|F|\log 2$ (because $1/m_F(A)>0$). Thus
$|\mathcal{C}|\le 2^{(\epsilon/2)|F|}$.

Let $C\in\mathcal{C}$. By condition~(b), the number of solutions in $C^r$
to the system $(F,A,b)$ is at most $\eta|F|^{r-k}+|Z|<
2\eta|F|^{r-k}$. Since $(F,A,b)$ is $f$-supersaturated, the definition of
$\eta$ means that $|C|\le {\rm ex}(F,A,b)+(\epsilon/2)|F|$.

The total number of solution-free subsets is at most
$|\mathcal{C}|2^{\max_{C\in\mathcal{C}}|C|}$. The inequalities just proved
mean this is at most $2^{{\rm ex}(F,A,b)+\epsilon|F|}$, as claimed.
\end{proof}

If $A$ is not abundant, then the conclusion of Theorem~\ref{thm:eqn_count}
need not hold. For example, let $A = (1, 1)$, $b = (0)$, and consider the
cyclic group $C_n$ for $n$ odd.
Observe that the pairs $(x,y)$ such that $x+y=0$ and $x \ne y$ partition $C_n
\setminus \{0\}$. Therefore $\mbox{ex}(C_n,A,b) = (n+1)/2$.
However, one can construct a solution-free set by including either $x$ or
$y$ or neither for each pair $(x,y)$, so there are at least $3^{(n-1)/2}$
solution-free sets.
There are similar examples with larger values of $k$ and $r>k+2$.
% As an example where $r > k+2$, the same behaviour holds for $C_n$ when $n$
% is prime and $A = \begin{pmatrix} 1&1&1&0&0\\0&0&0&1&1 \end{pmatrix}$.

Additionally, when $F=[N]$, the condition that~$A$ is fixed as
$|F|=N\to\infty$ is necessary. For example, for
the equation $w+x+(10N)y-(10N)z=N$, the maximum size of a solution-free subset
of~$[N]$ is $N/2$ (since for every pair $w,x\in[N]$ with $w+x=N$, a
solution-free set can include at most one of~$w$ or~$x$), but there are at
least $3^{(N-1)/2}$ solution free sets, since for every $w,x\in[N]$ with
$w+x=N$ and $w\ne x$, we can include either~$w$ or~$x$ or neither to form a
solution-free set.

We now turn to solution-free subsets within randomly chosen subsets
$X\subset F$, as mentioned in~\S\ref{subsec:lineq}. Here is
the main result.

\begin{thm}\label{thm:eqn_sparse}
  Let $k,r\in\mathbb{N}$ and let $f:\mathbb{R}^+\to\mathbb{R}^+$ be null.
  Let $(F,A,b,Z)$ be a $k \times r$ linear system with $A$ abundant and
  $(F,A,b)$ $f$-supersaturated. Given $\epsilon>0$, there exists
  $c=c(k,r,f,\epsilon)$ (or $c=c(A,f,\epsilon)$ in the case $F=[N]$) and
  $\eta=\eta(f,\epsilon)>0$, such that, if $|F|>c$ and $|Z|<\eta|F|^{r-k}$,
  $p \ge
  c|F|^{-1/m_F(A)}$, and $X \subset F$ is a random subset with each element
  included independently with probability~$p$, then the following event
  holds with probability greater than $1-\exp\{-\epsilon^3 p |F| / 512\}$:
$$
\mbox{every solution-free subset has at most $p(\mbox{\rm
    ex}(F,A,b)+\epsilon|F|)$ elements.}
$$
\end{thm}

As sketched earlier, when $F$ is a finite field the condition $p\ge
c|F|^{-1/m_F(A)}$ in Theorem~\ref{thm:eqn_sparse} is tight up to the value
of the constant $c$ appearing, at least under some mild restrictions on
$(F,A,b,Z)$. See R\"odl and Ruci\'nski~\cite{RR} for more detail.

To prove Theorem~\ref{thm:eqn_sparse} we use the following probabilistic
lemma from~\cite{ST2}. A very straightforward expectation argument applied
to Theorem~\ref{thm:eqn_cover} will give Theorem~\ref{thm:eqn_sparse} with
just a slightly worse bound on~$p$, namely $p \ge c|F|^{-1/m_F(A)}\log
|F|$; the point of the next lemma is that it allows us to take advantage of
condition~(a) of Theorem~\ref{thm:eqn_cover}, namely that $T\subset I$, to
remove the extra log factor and obtain a best possible result. The lemma is
stated in a generality that is not needed for the present application, but
we quote it as it appears in~\cite{ST2}, apart from replacing a tuple
$(T_1,\ldots,T_{s'})$ by a single set~$T$. The proof of the lemma is just
a combination of a Chernoff bound and the union bound.

\begin{lem}[{\cite[Lemma~10.3]{ST2}}]\label{lem:sparse}
  Given $0<\nu<1$ there is a constant $\phi=\phi(\nu)$ such that the
  following holds.  Let $L$ be a set, $|L|=n$, and let $\mathcal{I} \subset
  \mathcal{P}L$.  Let $t\ge1$, let $\phi t/n \le p \le1$ and let $\nu
  n/2\le d \le n$.  Suppose for each $I\in\mathcal{I}$ there exists both
  $T_I\subset I$ and $D=D(T_I)\subset L$, such that $|T_I| \le t$ and
  $|D(T_I)| \le d$.  Let $X\subset L$ be a random subset where each element
  is chosen independently with probability $p$. Then
\begin{equation*}
\mathbb{P}\left( |D(T_I) \cap X|>(1+\nu)pd\mbox{ for some }
  I \subset X,\, I \in \mathcal{I} \right)
  \le \exp\{ -\nu^2 pd/32 \}.
\end{equation*}
\end{lem}

\begin{proof}[Proof of Theorem~\ref{thm:eqn_sparse}]
  Let $L=F$ and let $\mathcal{I}$ be the set of solution-free sets for the
  system $(F,A,b,Z)$. Let $\eta>0$ be such that $f(2\eta)<\epsilon/4$,
  which exists because $f$ is null. Let $c'$ be the constant supplied by
  Theorem~\ref{thm:eqn_cover} when $\eta$ is used in place
  of~$\epsilon$. Then $(F,A,b,Z)$ satisfies the the conditions of the
  theorem (assuming as ever that $\eta<1/2$) so we obtain a collection
  $\mathcal{C}$ of containers for~$\mathcal{I}$.  For $I\in\mathcal{I}$, let
  $T=T_I$, $C=C(T)$ be as given by the theorem. Our aim is to apply
  Lemma~\ref{lem:sparse} with $D(T)=C(T)$ and
  \[
  \nu = \epsilon/2, \quad d = \mbox{ex}(F,A,b)+\epsilon|F|/4, 
  \quad t = c' |F|^{1-1/m_F(A)}\,.
  \]
  Note that, by condition~(b) of Theorem~\ref{thm:eqn_cover}, for each
  $C=C(T)$, $C^r-Z$ has at most $\eta |F|^{r-k}$ solutions to $Ax=b$,
  and so $C^r$ has at most $2\eta |F|^{r-k}$ solutions, and hence
  $|C(T)|\le d$ holds by the supersaturation property. Condition~(a)
  implies $|T_I|\le t$. The conditions of
  Lemma~\ref{lem:sparse} then hold with $n=|F|$, noting that $d\ge \nu n/2$
  and that $p\ge c|F|^{-1/m_F(A)}\ge \phi t/n$ if $c$ is large
  enough. Finally, note that each solution-free set $I\in\mathcal{I}$ is
  contained in $C(T_I)$ and $(1+\nu)pd \le p(\mbox{ex}(F,A,b)+\epsilon|F|)$,
  so the concluding inequality of the lemma means that the property of the
  theorem fails with probability bounded by
  \[
  \exp\{-\nu^2pd/32\} \le \exp\left\{-\epsilon^3 p|F| / 512
  \right\},
  \]
  completing the proof.
\end{proof}

\begin{proof}[Proof of Theorem~\ref{thm:szem_sparse}]
  Let $([N],A,b,Z)$ be the $(\ell-2)\times \ell$ linear system
  corresponding to forbidding an $\ell$-term arithmetic progression in
  $[N]$. For example if $\ell=3$ then $A=(1,1,-2)$, $b=(0)$ and $Z$ is the
  set of solutions of the form $x+x-2x=0$ that are discounted, so $|Z|=N$.
  As mentioned in~\S\ref{subsec:supersat}, Varnavides' theorem shows this
  system is $f$-supersaturated for some $f$ not depending on~$N$. It can
  readily be checked that $m_{[N]}(A)=\ell-1$. Since
  $\mbox{ex}(F,A,b)=o(N)$, the result immediately follows by applying
  Theorem~\ref{thm:eqn_sparse}.
\end{proof}

\section{Sidon sets}\label{sec:Sidon}

In this section we prove Theorem~\ref{thm:sidon}. We begin with the simple
construction giving the lower bound.

\begin{proof}[Proof of Theorem~\ref{thm:sidon}, lower bound]
  Suppose that $n=4p(p-1)$ for some prime~$p$.  Ruzsa~\cite{R} shows that
  there is a set $S \subset [p(p-1)]$ of size $p-1$ such that every sum of
  two elements of $S$ is distinct modulo $p(p-1)$.  Thus for any $U_1, U_2,
  U_3, U_4 \subset S$ satisfying $U_i \cap U_j = \emptyset$ for $i \ne j$,
  the set
  \[
  U_1 \cup (U_2 + p(p-1)) \cup (U_3 + 2p(p-1)) \cup (U_4 + 3p(p-1))
  \]
  is a Sidon subset of $[4p(p-1)]$, where $V + x := \{v + x : v \in V\}$.
  This gives $5^{p-1} = \sqrt{5}^{(1+o(1))\sqrt{n}} >
  2^{(1.16+o(1))\sqrt{n}}$ Sidon subsets of $[n]=[4p(p-1)]$.  The general
  case follows by embedding $[4p(p-1)]$ into $[n]$, where $p$ is the
  largest prime such that $4p(p-1) < n$, and using the fact that the ratio
  of successive primes tends to 1. (We note that any construction for large
  modular Sidon sets could have been used here; this includes the classical
  constructions of Singer~\cite{Sing} and of Bose~\cite{Bose}.)
\end{proof}

To prove the upper bound we construct, in the natural way, the hypergraph
representing the solutions to $w+x=y+z$ in a subset $S\subset[n]$. We then
apply an iterated version of Theorem~\ref{thm:cover} to this hypergraph in
an entirely mechanical way; all that is needed is to set appropriate values
and to check the conditions.

Corollary~\ref{cor:sparse_container} is an iterated version of
Theorem~\ref{thm:coverweak} but it is a bit too crude for use here. Only a
constant number of iterations are involved (of the order $\log(1/\epsilon)$)
whereas here the number of iterations is a function of~$n$, as the
containers shrink from size $n$ to order $\sqrt n$. Moreover we need to
take account of a change in behaviour of the codegree function when the
size of the container drops below $n^{2/3}$, as the dominant contribution
then comes from $\delta_2$ rather than $\delta_4$ (see
equation~(\ref{eqn:sidon_delta}) below); Kohayakawa, Lee, R\"odl and
Samotij~\cite{KLRS} noticed an interesting behavioural change at the same
point, for a closely related problem.

The version we need is the following. It is identical
to~\cite[Theorem~6.3]{ST2}, but with a tuple $(T_1,\ldots,T_s)$ replaced by
a single set~$T$.

\begin{thm}\label{thm:iteration}
  Let $G$ be an $r$-graph on vertex set $[n]$. Let $e_0 \le e(G)$. Suppose
  that, for each $U \subset [n]$ with $e(G[U]) \ge e_0$, the function
  $\tau(U)$ satisfies $\tau(U)<1/2$ and $\delta(G[U],\tau(U)) \le 1/12r!$.
  For $e_0\le m\le e(G)$ define
\begin{align*}
 f(m) &= \max\{\, - |U| \tau(U) \log \tau(U) : U \subset [n],\, e(G[U]) \ge
 m \} \\ 
\tau^* &= \max\{\, \tau(U) : U \subset [n],\, e(G[U]) \ge e_0 \}
\end{align*}
Let $k= \log(e_0/e(G))/\log(1-1/2r!)$.
Then there is a collection $\C \subset \mathcal{P}[n]$ such that
\begin{itemize}
\item[(a)] for every independent set $I$ there exists $T\subset I$ with $I
  \subset C(T) \in \C$ and $|T| \le 288(k+1)r!^2r\tau^* n$,
 \item[(b)] $e(G[C]) \le e_0$ for all $C \in \C$,
 \item[(c)] $\log|\C| \le 288r!^2r \sum_{0\le i< k} f( e_0 / (1-1/2r!)^i )$.
\end{itemize}
\end{thm}

\begin{proof}[Proof of Theorem~\ref{thm:sidon}, upper bound]
  Let $G$ be the $4$-graph on vertex set $[n]$, where $\{w,x,y,z\} \in
  [n]^{(4)}$ is an edge whenever $w+x = y+z$. Sidon sets correspond to
  independent sets in $G$ (although the converse is not always true, since
  solutions to $w+x=y+z$ where $w=x$ or $y=z$ do not correspond to edges
  of~$G$).

  Let $U\subset [n]$ and $u=|U|$. For $i \in [n-1]$, let $t_i = |\{\{x,y\}
  \in U^{(2)} : x<y, y-x = i\}|$.  Note that $\sum_i t_i = {u \choose 2}$.
  Each pair of sets $\{w,z\}\ne\{y,x\}$ with $w-z=y-x$ corresponds to an
  edge with $w+x=y+z$, and each such edge corresponds to the two pairs
  $\{w,z\}\ne\{y,x\}$ and $\{w,y\}\ne\{x,z\}$. Hence, for large $u$, the
  number of edges in $G[U]$ satisfies
  \begin{align*}
    m = e(G[U]) = \frac{1}{2} \sum_{i=1}^{n-1} {t_i \choose 2} \ge
    \frac{n-1}{2} { \frac{1}{n-1} \sum_i t_i \choose 2} \ge u^4 / 20n.
  \end{align*}

  We shall apply Theorem~\ref{thm:iteration} to the graph $G$. To this end,
  let $\beta = 3\times10^{14}$, let $e_0 = \beta^4 n / 20$, and consider $U
  \subset [n]$ where $e(G[U])\ge e_0$. Since the bound we are proving is an
  asymptotic one, we may assume that $n$ is large: this in turn means that
  $e_0$ is large, so $m$ is large and (since $m\le {u\choose 4}$) $u$ is
  also large; in particular the inequality $u\le (20nm)^{1/4}$ always
  holds.

  Let $k=12r!=288$. Now put
  \[
  \tau = \tau(U)= \max\{24ku^2/m, (4ku/m)^{1/3}\}.
  \]
  To apply Theorem~\ref{thm:iteration} we must check that $\tau \le 1/2$
  and that $\delta \le 1/12r!$. For convenience we shall verify $\tau\le
  1/12$ (in fact $\tau$ is far smaller).

  Recall the definition of $d^{(j)}(w)$. In $G[U]$, observe that
  $d^{(2)}(w) \le u/2 + u = 3u/2$, since for $x \in U$, the number of
  solutions of the form $w+x = y+z$ is at most $u/2$ and the number of
  solutions of the form $w+y = x+z$ is at most $u$; similarly $d^{(3)}(w)
  \le 3$ and $d^{(4)}(w) \le 1$.

  Hence
  \[
  \delta_2 \le \frac{3u^2}{8\tau m} \qquad \delta_3 \le \frac{3u}{4\tau^2
    m} \qquad \delta_4 \le \frac{u}{4\tau^3 m},
  \]
  and (since $\tau < 1/12$, as we shall check shortly)
  \begin{equation}\label{eqn:sidon_delta}
    \delta = 32\delta_2 + 16\delta_3 + 4\delta_4
    \le \frac{12u^2}{\tau m} + \frac{2u}{\tau^3 m}.
  \end{equation}
  Then both terms on the right hand side of~(\ref{eqn:sidon_delta}) are
  less than $1/2k$, so $\delta \le 1/12r!$ is satisfied.

  If $\tau \le 24ku^2/m$, then the constraint $\tau \le 1/12$ holds
  comfortably (since $u\le (20nm)^{1/4}$ and $m\ge e_0$), and furthermore
  \begin{align*}
    u \tau \log(1/\tau)
    &\le (24ku^3/m) \log( m / (24ku^2) ) \\
    &\le 20^{3/4} 48k \sqrt{n} \left(\frac{n}{m}\right)^{1/4} \log \frac{
      (m/n)^{1/4} }{ (24k)^{1/2}(20)^{1/4} } \\ 
    &=: f_1(m),
  \end{align*}
  where the first inequality holds since $\tau \log (1/\tau)$ is an
  increasing function of $\tau$ when $\tau < 1/e$, and the second
  inequality holds since $u^3 \log (m/(24ku^2))$ is an increasing function
  of $u$ when $u \le e^{-1/3} \sqrt{m/24k}$, and $u \le (20nm)^{1/4}$ which
  is less than $e^{-1/3} \sqrt{m/24k}$ because $m\ge e_0$.

  Alternatively, if $\tau \le (4ku/m)^{1/3}$ then the constraint $\tau \le
  1/12$ is easily satisfied, and also
  \begin{align*}
    u \tau \log(1/\tau)
    &\le (4ku^4/27m)^{1/3} \log( m/4ku ) \\
    &\le 6k^{1/3} n^{1/3} \log n \\
    &=: f_2(m) \qquad\mbox{when $m \le e(G)$},
  \end{align*}
  where the second inequality holds because $u^{4/3} \log( m/4ku)$ is an
  increasing function of $u$ for $u \le e^{-3/4} m/4k$ (which is larger
  than $(20nm)^{1/4}$), together with the bound $m \le n^4$. Let $f_2(m)=0$
  for $m > e(G)$.

  Therefore the conditions of Theorem~\ref{thm:iteration} are
  satisfied. Moreover, since $f_1$ and $f_2$ are non-increasing functions
  of $m$, we have $f(m) \le \max\{ f_1(m), f_2(m) \}$ for $m \ge e_0$. So
  let $\C$ be the collection of containers given by
  Theorem~\ref{thm:iteration} for the graph $G$, where each $C \in \C$
  satisfies $e(G[C]) \le e_0$. Writing $\alpha = 1-1/2r!$ and $m_i = e_0 /
  \alpha^i = \beta^4 n/20\alpha^i$, we have $\log|\C| \le 288rr!^2 \sum_{i
    \ge 0} f( m_i )$.

  Note that $\sum_{i \ge 0} \gamma^i = 1/(1-\gamma)$ and $\sum_{i \ge 0} i
  \gamma^i = \gamma / (1-\gamma)^2$, so
  \begin{align*}
    288rr!^2 \sum_{i\ge0} f_1(m_i)
    &= 288rr!^2 20^{3/4} 48k  \sqrt{n} \sum_{i\ge0}
    \frac{(20\alpha^i)^{1/4}}{\beta} \log \frac{
      \beta}{\alpha^{i/4}\sqrt{480k} } \\ 
    &= 288rr!^2 \frac{960k \sqrt{n}}{\beta} \left( \frac{\alpha^{1/4}\log
        (1/\alpha)}{4(1-\alpha^{1/4})^2}
      + \frac{\log (\beta / \sqrt{480k})}{1-\alpha^{1/4}} \right) \\
    &< \frac{7\sqrt{n}}{2}.
  \end{align*}
  Observe that $m_i \ge n^4 > e(G)$ when $i \ge 3 \log n / \log(1/\alpha)$
  (and hence $f_2(m_i) = 0$), so
  \[
  \sum_{i \ge 0} f_2(m_i) = o(\sqrt{n}).
  \]

  Each Sidon set in $[n]$ is a subset of size at most $(1+o(1))\sqrt{n}$ of
  some $C \in \C$, where $|C| \le u_0=\beta\sqrt n$ (because if $|C| > u_0$
  then $e(G[C]) \ge 20u^4/n>e_0$). The number of such subsets is at most
  $\binom{\beta\sqrt n}{(1+o(1))\sqrt n}$. Using the standard inequality
  $\binom{n}{k}\le \left(\frac{en}{k}\right)^k$, the number of these
  subsets is at most $\exp\{(1+\log \beta + o(1))\sqrt n\}$. Letting
  $\mathcal{S}$ be the collection of Sidon subsets of $[n]$,
  \begin{align*}
    \frac{\log |\mathcal{S}|}{\sqrt{n}} &\le
    1 + \log \beta + o(1) +
    \frac{288rr!^2}{\sqrt{n}} \sum_{i \ge 0} f_1(m_i) +
    \frac{288rr!^2}{\sqrt{n}} \sum_{i \ge 0} f_2(m_i) \\ 
    &< 1+\log\beta + 7/2 +o(1) \,< \,55\log 2+o(1),
  \end{align*}
  which completes the verification.
\end{proof}

\paragraph*{\bf Acknowledgement} We are grateful to a referee for a very
careful reading of the manuscript and for the suggestion that the
paper be made more self-contained, which led to the inclusion of
Theorem~\ref{thm:coverweak}.

\end{document}